\documentclass[11pt]{amsart}

\usepackage{amsmath,amssymb,latexsym,amsthm,newlfont,enumerate}
\usepackage{etoolbox,changepage}
\usepackage{hyperref}

\usepackage[all]{xy}
\usepackage{pstricks}
\usepackage{color}

\theoremstyle{plain}

\newtheorem{thm}{Theorem}[section]

\newtheorem{pro}[thm]{Proposition}

\newtheorem{lem}[thm]{Lemma}

\newtheorem{cor}[thm]{Corollary}
\newtheorem{con}[thm]{Conjecture}
\newtheorem*{GenAbundCon}{Generalised Abundance Conjecture}
\newtheorem*{GenNonvanCon}{Generalised Nonvanishing Conjecture}

\newtheorem*{AbundCon}{Abundance Conjecture}
\newtheorem*{NonvanCon}{Nonvanishing Conjecture}
\newtheorem*{SemCon}{Semiampleness Conjecture on Calabi-Yau pairs}
\newtheorem*{SemConVar}{Semiampleness Conjecture on Calabi-Yau varieties}

\newtheorem{thmA}{Theorem}

\newtheorem{corA}[thmA]{Corollary}

\theoremstyle{definition}

\newtheorem{dfn}[thm]{Definition}

\newtheorem{rem}[thm]{Remark}

\newtheorem{exa}[thm]{Example}

\theoremstyle{remark}

\newcommand{\N}{\mathbb{N}}

\newcommand{\R}{\mathbb{R}}
\newcommand{\Q}{\mathbb{Q}}
\newcommand{\PS}{\mathbb{P}}

\newcommand{\OO}{\mathcal{O}}

\DeclareMathOperator{\codim}{codim}

\DeclareMathOperator{\mult}{mult}

\DeclareMathOperator{\Supp}{Supp}

\DeclareMathOperator{\Pic}{Pic}

\DeclareMathOperator{\NEb}{\overline{\mathrm{NE}}}

\begin{document}
\title[On Generalised Abundance, I]{On Generalised Abundance, I}

\author{Vladimir Lazi\'c}
\address{Fachrichtung Mathematik, Campus, Geb\"aude E2.4, Universit\"at des Saarlandes, 66123 Saarbr\"ucken, Germany}
\email{lazic@math.uni-sb.de}

\author{Thomas Peternell}
\address{Mathematisches Institut, Universit\"at Bayreuth, 95440 Bayreuth, Germany}
\email{thomas.peternell@uni-bayreuth.de}

\thanks{
Lazi\'c was supported by the DFG-Emmy-Noether-Nachwuchsgruppe ``Gute Strukturen in der h\"oherdimensionalen birationalen Geometrie". Peternell was supported by the DFG grant ``Zur Positivit\"at in der komplexen Geometrie". We would like to thank S.\ Schreieder for asking a question about the paper \cite{LP18} at a seminar in Munich in December 2017 which motivated this paper, and B.\ Lehmann and L.\ Tasin for useful discussions.
}

\begin{abstract}
The goal of this paper is to make a surprising connection between several central conjectures in algebraic geometry: the Nonvanishing Conjecture, the Abundance Conjecture, and the Semiampleness Conjecture for nef line bundles on $K$-trivial varieties.
\end{abstract}

\maketitle
\setcounter{tocdepth}{1}
\tableofcontents

\section{Introduction}

The goal of this paper is to establish a surprising connection between several central conjectures in algebraic geometry: the \emph{Nonvanishing Conjecture}, the \emph{Abundance Conjecture}, and the \emph{Semiampleness Conjecture} for nef line bundles on $K$-trivial varieties.

\medskip

We briefly recall these conjectures; note that an important special case of klt pairs are varieties with canonical singularities.

\begin{NonvanCon} 
Let $(X,\Delta)$ be a projective klt pair such that $K_X + \Delta$ is pseudoeffective. Then there exists a positive integer $m$ such that 
$$ H^0\big(X,\OO_X(m(K_X+\Delta))\big) \neq 0.$$
\end{NonvanCon} 

If additionally $K_X+\Delta$ is nef, then the Abundance Conjecture predicts that $K_X+\Delta$ is semiample:
 
\begin{AbundCon} 
Let $(X,\Delta)$ be a klt pair such that $K_X + \Delta$ is nef. Then $K_X + \Delta$ is semiample, i.e.\ some multiple $m(K_X + \Delta)$ is basepoint free. 
\end{AbundCon}

On the other hand, if $(X,\Delta)$ be a projective klt pair such that $K_X+\Delta$ is numerically trivial, i.e.\ $(X,\Delta)$ is a \emph{Calabi-Yau pair}, and if $L$ is a nef Cartier divisor on $X$, then the Semiampleness Conjecture (sometimes referred to as SYZ conjecture) predicts that the numerical class of $L$ contains a semiample divisor:

\begin{SemCon} 
Let $(X,\Delta)$ be a projective klt pair such that $K_X + \Delta \equiv 0 $. Let $L$ be a nef Cartier divisor on $X$. Then there exists a semiample $\Q$-divisor $L'$ such that $L \equiv L'$.
\end{SemCon} 

We will refer to this conjecture in the sequel simply as the \emph{Semiampleness Conjecture}. 

\medskip

The Nonvanishing and Abundance Conjectures were shown in dimension $3$ by the efforts of various mathematicians, in particular Miyaoka, Kawamata, Koll\'ar, Keel, Matsuki, M\textsuperscript{c}Kernan \cite{Miy87,Miy88a,Miy88b,Kaw92,Kol92,KMM94}. In arbitrary dimension, the abundance for klt pairs of log general type was proved by Shokurov and Kawamata \cite{Sho85,Kaw85b}, and the abundance for varieties with numerical dimension $0$ was established by Nakayama \cite{Nak04}. 

In contrast, not much is known on the Semiampleness Conjecture in dimensions at least $3$; we refer to \cite{LOP16,LOP16a} for the state of the art, in particular on the important contributions by Wilson on Calabi-Yau threefolds and by Verbitsky on hyperk\"ahler manifolds.

\medskip

In \cite{LP18} we made progress on these conjectures, especially on the Nonvanishing Conjecture. The following conjectures grew out of our efforts to push the possibilities of our methods to their limits.

\begin{GenNonvanCon}
Let $(X,\Delta)$ be a klt pair such that $K_X+\Delta$ is pseudoeffective. Let $L$ be a nef $\Q$-divisor on $X$. Then for every $t\geq0$ the numerical class of the divisor $K_X+\Delta+tL$ belongs to the effective cone.
\end{GenNonvanCon}

\begin{GenAbundCon}\label{con:genAbundanceSemiample}
Let $(X,\Delta)$ be a projective klt pair such that $K_X+\Delta$ is pseudoeffective and let $L$ be a nef Cartier divisor on $X$. If $K_X+\Delta+L$ is nef, then there exists a semiample $\Q$-divisor $M$ such that $K_X+\Delta+L\equiv M$.
\end{GenAbundCon}

These conjectures generalise the Nonvanishing and Abundance conjectures (by setting $L=0$) and the Semiampleness Conjecture (by setting $K_X+\Delta\equiv 0$), and finally make and explains precisely the relationship between them. A careful reader will have noticed that the conclusions of the Nonvanishing and Abundance conjectures are a priori stronger: that there is an effective (respectively semiample) divisor which is \emph{$\Q$-linearly equivalent} to $K_X+\Delta$, and not merely \emph{numerically equivalent}. However, the Nonvanishing for adjoint bundles is of numerical character \cite{CKP12}, and the same is true for the  Abundance \cite{Fuk11}.

\medskip

The Generalised Abundance Conjecture should be viewed as an extension of the Basepoint Free Theorem of Shokurov and Kawamata, in which the nef divisor $L$ is additionally assumed to be big.

\medskip

A weaker version of the Generalised Abundance Conjecture was stated by Koll\'ar \cite[Conjecture 51]{Kol15} in the context of the existence of elliptic fibrations on Calabi-Yau manifolds and the Semiampleness Conjecture. A more general version of the Generalised Nonvanishing Conjecture (without assuming pseudoeffectivity of $K_X+\Delta$) was posed as a question in  \cite[Question 3.5]{BH14}, with an expected negative answer. In this paper and its sequel we argue quite the opposite.

\medskip

\noindent{\sc The content of the paper.}
Here, we show that \emph{the Generalised Nonvanishing and Abundance Conjectures follow from the conjectures of the Minimal Model Program and from the Semiampleness Conjecture}. 

\medskip

Throughout this work, all varieties are complex, normal and projective. We denote the numerical dimension of $K_X+\Delta$ by $\nu(X,K_X+\Delta)$, see Definition \ref{dfn:kappa}. 

The following are the main results of this paper.

\begin{thmA}\label{main_theorem1}
Assume the termination of flips in dimensions at most $n$ and the Abundance Conjecture in dimensions at most $n$. 

Let $(X,\Delta)$ be a klt pair of dimension $n$ such that $K_X+\Delta$ is pseudoeffective, and let $L$ be a nef $\Q$-divisor on $X$.
\begin{enumerate}
\item[(i)] If $\nu(X,K_X+\Delta)>0$, then for every $t\geq0$ the numerical class of the divisor $K_X+\Delta+tL$ belongs to the effective cone. 
\item[(ii)] Assume additionally the Semiampleness Conjecture in dimensions at most $n$. If $\nu(X,K_X+\Delta)=0$, then for every $t\geq0$ the numerical class of the divisor $K_X+\Delta+tL$ belongs to the effective cone. 
\end{enumerate}
\end{thmA}

\begin{thmA}\label{main_theorem2}
Assume the termination of flips in dimensions at most $n$ and the Abundance Conjecture in dimensions at most $n$. 

Let $(X,\Delta)$ be a klt pair of dimension $n$ such that $K_X+\Delta$ is pseudoeffective, and let $L$ be a nef $\Q$-divisor on $X$ such that $K_X+\Delta+L$ is nef.
\begin{enumerate}
\item[(i)] If $\nu(X,K_X+\Delta)>0$, then there exists a semiample $\Q$-divisor $M$ on $X$ such that $K_X+\Delta+L\equiv M$.
\item[(ii)] Assume additionally the Semiampleness Conjecture in dimensions at most $n$. If $\nu(X,K_X+\Delta)=0$, then there exists a semiample $\Q$-divisor $M$ on $X$ such that $K_X+\Delta+L\equiv M$.
\end{enumerate}
\end{thmA}

As a consequence, we obtain the following new results in dimensions two and three.

\begin{corA}\label{cor:dim2}
Let $(X,\Delta)$ be a klt pair of dimension $2$ such that $K_X+\Delta$ is pseudoeffective, and let $L$ be a nef $\Q$-divisor on $X$. Then for every $t\geq0$ the numerical class of the divisor $K_X+\Delta+tL$ belongs to the effective cone.  If additionally $K_X+\Delta+L$ is nef, then there exists a semiample $\Q$-divisor $M$ on $X$ such that $K_X+\Delta+L\equiv M$.
\end{corA}

\begin{corA}\label{cor:dim3}
Let $(X,\Delta)$ be a klt pair of dimension $3$ such that $K_X+\Delta$ is pseudoeffective, and let $L$ be a nef $\Q$-divisor on $X$. Assume that $\nu(X,K_X+\Delta)>0$. Then for every $t\geq0$ the numerical class of the divisor $K_X+\Delta+tL$ belongs to the effective cone.  If additionally $K_X+\Delta+L$ is nef, then there exists a semiample $\Q$-divisor $M$ on $X$ such that $K_X+\Delta+L\equiv M$.
\end{corA}

There are a few unresolved cases of the Semiampleness Conjecture in dimension $3$, see \cite{LOP16a}, hence one cannot currently derive the Generalised Nonvanishing Conjecture and the Generalised 
Abundance Conjecture 
unconditionally on threefolds when $\nu(X,K_X+\Delta)=0$. 

\medskip

As an intermediate step in the proofs, we obtain as a consequence of Corollary \ref{cor:genAbundancereduction1} the following reduction for the Generalised Nonvanishing Conjecture.
 
\begin{thmA}\label{thm:reduction}
Assume the termination of klt flips in dimension $n$. If the Generalised Nonvanishing Conjecture holds for $n$-dimensional klt pairs $(X,\Delta)$ and for nef divisors $L$ on $X$ such that $K_X + \Delta + L$ is nef, then the Generalised Nonvanishing Conjecture holds in dimension $n$. 
\end{thmA} 

\medskip

\noindent{\sc Comments on the proof.}
Our first step is to show that, modulo the Minimal Model Program (MMP), we may assume that $K_X+\Delta+mL$ is nef for $m\gg0$; this argument is essentially taken from \cite{BH14,BZ16}. The essential ingredient in the proof is the boundedness of extremal rays, which forces any $(K_X+\Delta+mL)$-MMP to be $L$-trivial. The details are in Proposition \ref{pro:contocon}.

\medskip

The main problem in the proofs of Theorems \ref{main_theorem1} and \ref{main_theorem2} -- presented in Section \ref{sec:main} -- was to find a correct inductive statement, and the formulation of \cite[Lemma 4.2]{GL13} was a great source of inspiration. The key is to consider the behaviour of the positive part $P_\sigma(K_X+\Delta+L)$ in the Nakayama-Zariski decomposition, see \S\ref{subsec:numdim} below. Then our main technical result, Theorem \ref{thm:semiampleP_sigma}, shows that the positive part of the pullback of $K_X+\Delta+L$ to some smooth birational model of $X$ is semiample (up to numerical equivalence). This implies Theorems \ref{main_theorem1} and \ref{main_theorem2} immediately.

As a nice corollary, we obtain that, modulo the MMP and the Semiampleness Conjecture, the section ring of $K_X+\Delta+L$ is finitely generated, up to numerical equivalence. This is Corollary \ref{cor:fingen}.

\medskip

In order to prove Theorem \ref{thm:semiampleP_sigma}, we may assume additionally that $K_X+\Delta+mL$ is nef for $m\gg0$. Then we distinguish two cases. If the \emph{nef dimension} of $K_X+\Delta+mL$ is maximal, then we run a carefully chosen MMP introduced in \cite{KMM94} to show that $K_X+\Delta+L$ is also big; we note that this is the only place in the proof where the Semiampleness Conjecture is used. Once we know the bigness, we conclude by running the MMP as in \cite{BCHM}. 

If the nef dimension of $K_X+\Delta+mL$ is not maximal, the proof is much more involved. We use the \emph{nef reduction map} of $K_X+\Delta+mL$ to a variety of lower dimension. This map has the property that both $K_X+\Delta$ and $L$ are numerically trivial on its general fibres. Then we use crucially Lemma \ref{lem:lehmann} below, which shows that then $L$ is (birationally) a pullback of a divisor on the base of the nef reduction. This result is a consequence of previous work of Nakayama and Lehmann \cite{Nak04,Leh15}, and is of independent interest. We conclude by running a relative MMP over the base of the nef reduction and applying induction on the dimension. Here, the techniques of \cite{Nak04} and the main result of \cite{Amb05a} are indispensable.

We further note that, as stated in Theorems \ref{main_theorem1} and \ref{main_theorem2}, the Semiampleness Conjecture is used only when the numerical dimension of $K_X+\Delta$ is zero; otherwise, the conjectures follow only from the standard conjectures of the MMP (the termination of flips and the Abundance Conjecture). 

\medskip

\noindent{\sc Final remarks.}
It is natural to wonder whether the Generalised Nonvanishing and the Generalised Abundance Conjectures hold under weaker assumptions. In Section \ref{AppB} we give examples showing that the Generalised Abundance Conjecture fails in the log canonical setting, and also when $K_X+\Delta$ is not pseudoeffective. We also give an application towards Serrano's Conjecture.

Further, in Section \ref{sec:semiampleness} we reduce the Semiampleness Conjecture (modulo the MMP) to a much weaker version, which deals only with Calabi-Yau varieties with terminal singularities. 

Finally, in Section \ref{sec:lowdimensions} we discuss in detail the Semiampleness Conjecture in dimensions $2$ and $3$.

\section{Preliminaries}

A \emph{fibration} is a projective surjective morphism with connected fibres between two normal varieties.

We write $D \geq 0$ for an effective $\Q$-divisor $D$ on a normal variety $X$. If $f\colon X\to Y$ is a surjective morphism of normal varieties and if $D$ is an effective $\Q$-divisor on $X$, then $D$ is \emph{$f$-exceptional} if $\codim_Y f(\Supp D) \geq 2$.

A \emph{pair} $(X,\Delta)$ consists of a normal variety $X$ and a Weil $\Q$-divisor $\Delta\geq0$ such that the divisor $K_X+\Delta$ is $\Q$-Cartier. The standard reference for the foundational definitions and results on the singularities of pairs and the Minimal Model Program is \cite{KM98}, and we use these freely in this paper. We recall additionally that flips for klt pairs exist by \cite[Corollary 1.4.1]{BCHM}.

In this paper we use the following definitions:
\begin{enumerate}
\item[(a)] a $\Q$-divisor $L$ on a projective variety $X$ is \emph{num-effective} if the numerical class of $L$ belongs to the effective cone of $X$,\footnote{We would prefer the natural term \emph{numerically effective}; unfortunately, this would likely cause confusion with the term \emph{nef}.}
\item[(b)] a $\Q$-divisor $L$ on a projective variety $X$ is \emph{num-semiample} if there exists a semiample $\Q$-divisor $L'$ on $X$ such that $L\equiv L'$.
\end{enumerate}

\subsection{Models}

We recall the definition of negative maps and good models.

\begin{dfn}
Let $X$ and $Y$ be $\Q$-factorial varieties, and let $D$ be a $\Q$-divisor on $X$. A birational contraction $f\colon X\dashrightarrow Y$ is \emph{$D$-negative} if there exists a resolution $(p,q)\colon W\to X\times Y$ of the map $f$ such that $p^*D=q^*f_*D+E$, where $E\geq0$ is a $q$-exceptional $\Q$-divisor and $\Supp E$ contains the proper transform of every $f$-exceptional divisor. If additionally $f_*D$ is semiample, the map $f$ is a \emph{good model} for $D$.
\end{dfn}

Note that if $(X,\Delta)$ is a klt pair, then it has a good model if and only if there exists a Minimal Model Program with scaling of an ample divisor which terminates with a good model of $(X,\Delta)$, cf.\  \cite[Propositions 2.4 and 2.5]{Lai11}.

\subsection{Numerical dimension and Nakayama-Zariski decomposition}\label{subsec:numdim}

Recall the definition of the numerical dimension from \cite{Nak04,Kaw85}.

\begin{dfn}\label{dfn:kappa}
Let $X$ be a normal projective variety, let $D$ be a pseudoeffective $\Q$-Cartier $\Q$-divisor on $X$ and let $A$ be any ample $\Q$-divisor on $X$. Then the {\em numerical dimension\/} of $D$ is
$$\nu(X,D)=\sup\big\{k\in\N\mid \limsup_{m\rightarrow\infty}h^0(X, \mathcal O_X(\lfloor mD\rfloor+A))/m^k >0\big\}.$$
When the divisor $D$ is nef, then equivalently
$$\nu(X,D)=\sup\{k\in\N\mid D^k\not\equiv0\}.$$
If $D$ is not pseudoeffective, we set $\nu(X,D)=-\infty$.
\end{dfn}

The Kodaira dimension and the numerical dimension behave well under proper pullbacks: if $D$ is a $\Q$-Cartier $\Q$-divisor on a normal variety $X$, and if $f\colon Y\to X$ is a proper surjective morphism from a normal variety $Y$, then 
$$\kappa(X,D)=\kappa(Y,f^*D)\quad\text{and}\quad\nu(X,D)=\nu(Y,f^*D);$$ 
and if, moreover, $f$ is birational and $E$ is an effective $f$-exceptional divisor on $Y$, then
\begin{equation}\label{eq:compare}
\kappa(X,D)=\kappa(Y,f^*D+E)\quad\text{and}\quad\nu(X,D)=\nu(Y,f^*D+E).
\end{equation}
The first three relations follow from \cite[Lemma II.3.11, Proposition V.2.7(4)]{Nak04}. For the last one, one can find ample divisors $H_1$ and $H_2$ on $Y$ and ample divisor $A$ on $X$ such that $H_1\leq f^*A\leq H_2$, so that for $m$ sufficiently divisible we have
\begin{multline*}
h^0\big(Y,m(f^*D+E)+H_1\big)\leq h^0\big(Y,m(f^*D+E)+f^*A\big)\\
=h^0\big(X,m(D+E)+A\big)\leq h^0\big(Y,m(f^*D+E)+H_2\big).
\end{multline*} 
The desired equality of numerical dimensions follows immediately.

In particular, the Kodaira dimension and the numerical dimension of a divisor $D$ are preserved under any $D$-negative birational map.

Finally, if $D$ is a pseudoeffective $\R$-Cartier $\R$-divisor on a smooth projective variety $X$, we denote by $P_\sigma(D)$ and $N_\sigma(D)$ the $\R$-divisors forming the Nakayama-Zariski decomposition of $D$, see\ \cite[Chapter III]{Nak04} for the definition and the basic properties. We use the decomposition
$$ D = P_\sigma (D) + N_\sigma (D)$$  
crucially in the proofs in this paper.

\begin{lem}\label{lem:1}
Let $X$, $X'$, $Y$ and $Y'$ be normal varieties, and assume that we have a commutative diagram
\[
\xymatrix{ 
X' \ar[d]_{\pi'} \ar[r]^{f'} & Y' \ar[d]^{\pi}\\
X \ar[r]_{f} & Y,
}
\]
where $\pi$ and $\pi'$ are projective birational. Let $D$ be a $\Q$-Cartier $\Q$-divisor on $X$ and let $E\geq0$ be a $\pi'$-exceptional $\Q$-Cartier $\Q$-divisor on $X'$. If $F$ and $F'$ are general fibres of $f$ and $f'$, respectively, then 
$$\nu(F,D|_F)=\nu\big(F',(\pi'^*D+E)|_{F'}\big).$$
\end{lem}

\begin{proof}
We may assume that $F'=\pi'^{-1}(F)$. Then it is immediate that the map $\pi'|_{F'}\colon F'\to F$ is birational and the divisor $E|_{F'}$ is $\pi'|_{F'}$-exceptional. Hence, the lemma follows from \eqref{eq:compare}.
\end{proof}

We need several properties of the Nakayama-Zariski decomposition.

\begin{lem}\label{lem:Nakayama1}
Let $f\colon X\to Y$ be a surjective morphism from a smooth projective variety $X$ to a normal projective variety $Y$, and let $E$ be an effective $f$-exceptional divisor on $X$. Then for any pseudoeffective $\Q$-Cartier divisor $L$ on $Y$ we have 
$$P_\sigma(f^*L+E)=P_\sigma(f^*L).$$
\end{lem}

\begin{proof} 
If $Y$ is smooth, this is \cite[Lemma III.5.14]{Nak04}, and in general, this a special case of \cite[Lemma 2.16]{GL13}. We include the proof for the benefit of the reader.

By \cite[Lemma III.5.1]{Nak04} there exists a component $\Gamma$ of $E$ such that $E|_\Gamma$ is not $(f|_\Gamma)$-pseudoeffective, and hence $(f^*L+E)|_\Gamma$ is not $(f|_\Gamma)$-pse\-u\-do\-ef\-fec\-ti\-ve. In particular, $(f^*L+E)|_\Gamma$ is not pseudoeffective. This shows that $\Gamma\subseteq N_\sigma(f^*L+E)$ by \cite[Lemma III.1.7(3)]{Nak04}. 

Set 
$$E':=\sum_P \min\{\mult_P E,\mult_P N_\sigma(f^*L+E)\}\cdot P.$$
Since $E'\leq N_\sigma(f^*L+E)$, we have
\begin{equation}\label{eq:nsigma}
N_\sigma(f^*L+E-E')=N_\sigma(f^*L+E)-E'
\end{equation}
by \cite[Lemma III.1.8]{Nak04}. If $E'\neq E$, then $E-E'$ is $f$-exceptional, hence by the paragraph above there exists a component $\Gamma'$ of $E-E'$ such that $\Gamma'\subseteq N_\sigma(f^*L+E-E')$, which together with \eqref{eq:nsigma} gives a contradiction. Therefore, $E'=E$, and hence $E\leq N_\sigma(f^*L+E)$. This implies $N_\sigma(f^*L)=N_\sigma(f^*L+E)-E$ by \eqref{eq:nsigma}, and thus the lemma.
\end{proof}

\begin{lem}\label{lem:Nakayama2}
Let $f\colon X\to Y$ be a surjective morphism between smooth projective varieties and let $D$ be a pseudoeffective $\Q$-divisor on $Y$. If $P_\sigma(D)$ is nef, then $P_\sigma(f^*D)=f^*P_\sigma(D)$.
\end{lem}

\begin{proof}
This is \cite[Corollary III.5.17]{Nak04}: on the one hand, we have 
$$N_\sigma(f^*D)\geq f^*N_\sigma(D)$$ 
by \cite[Theorem III.5.16]{Nak04}. On the other hand, $f^*P_\sigma(D)$ is a nef divisor, hence movable,  such that $f^*P_\sigma(D)\leq f^*D$, and therefore $N_\sigma(f^*D)\leq f^*N_\sigma(D)$ by \cite[Proposition III.1.14(2)]{Nak04}.
\end{proof}

\begin{lem}\label{lem:Nakayama3}
Let $f\colon X\to Y$ be a surjective morphism from a smooth projective variety $X$ to a normal projective variety $Y$, and let $D$ be a pseudoeffective $\Q$-divisor on $Y$. If $P_\sigma(f^*D)$ is nef, then there is a birational morphism $\pi\colon Z \to Y$ from a smooth projective variety $Z$ such that $P_\sigma(\pi^*D)$ is nef.
\end{lem}

\begin{proof}
When $Y$ is smooth, this is \cite[Corollary III.5.18]{Nak04}. When $Y$ is only normal, we reduce to the smooth case: Let $\theta\colon Y'\to Y$ be a desingularisation and let $X'$ be a desingularisation of the main component of the fibre product $X\times_Y Y'$ with the induced morphisms $f'\colon X'\to Y'$ and $\theta'\colon X'\to X$.
\[
\xymatrix{ 
X' \ar[d]_{\theta'} \ar[r]^{f'} & Y' \ar[d]^{\theta}\\
X\ar[r]_f & Y,
}
\]
Denote $D':=\theta^*D$. Then we have
$$P_\sigma(f'^*D')=P_\sigma(f'^*\theta^*D)=P_\sigma(\theta'^*f^*D)=\theta'^*P_\sigma(f^*D),$$
where the last equality follows from Lemma \ref{lem:Nakayama2}. In particular, $P_\sigma(f'^*D')$ is nef, hence by \cite[Corollary III.5.18]{Nak04} there exists a birational morphism $\pi'\colon Z' \to Y'$ from a smooth projective variety $Z$ such that $P_\sigma(\pi'^*D')$ is nef. The lemma follows by setting $\pi:=\theta\circ\pi'$.
\end{proof}

\subsection{MMP with scaling of a nef divisor} \label{subsec:scaling}

In this paper, we use the Minimal Model Program with scaling of a \emph{nef divisor}; when the nef divisor is actually ample, this is \cite[\S3.10]{BCHM}. The general version is based on the following lemma \cite[Lemma 2.1]{KMM94}.

\begin{lem}\label{lem:KMM94}
Let $(X,\Delta)$ be a $\Q$-factorial klt pair such that $K_X + \Delta$ is not nef and let $L$ be a nef $\Q$-divisor on $X$. Let
$$\lambda=\sup\{t\mid t(K_X + \Delta) + L \text{ is nef}\,\}.$$
Then $\lambda\in\Q$ and there is a $(K_X + \Delta)$-negative extremal ray $R$ such that
$$\big(\lambda(K_X + \Delta) + L)\cdot R=0.$$
\end{lem}

We now explain how to run the MMP with scaling of a nef divisor; we use this procedure crucially twice in the remainder of the paper.

We start with a klt pair $(X,\Delta)$ such that $K_X+\Delta$ is not nef and a nef $\Q$-divisor $L$ on $X$. We now construct a sequence of divisorial contractions or flips 
$$\varphi_i\colon (X_i,\Delta_i)\dashrightarrow (X_{i+1},\Delta_{i+1})$$ 
of a $(K_X+\Delta)$-MMP with $(X_0,\Delta_0):=(X,\Delta)$, a sequence of divisors $L_i:=(\varphi_{i-1}\circ\dots\circ\varphi_0)_*L$ on $X_i$, and a sequence of rational numbers $\lambda_i$ for $i\geq0$ with the following properties:
\begin{enumerate}
\item[(a)] the sequence $\{\lambda_i\}_{i\geq0}$ is non-decreasing,
\item[(b)] each $\lambda_i(K_{X_i}+\Delta_i)+L_i$ is nef and $\varphi_i$ is $\big(\lambda_i(K_{X_i}+\Delta_i)+L_i\big)$-trivial,
\item[(c)] the map $\varphi_i\circ\dots\circ\varphi_0$ is $\big(s(K_X+\Delta)+L)$-negative for $s>\lambda_i$,
\end{enumerate}

Indeed, we set $\lambda_0:=\sup\{t\mid t(K_X + \Delta) + L \text{ is nef}\}$. Then by Lemma \ref{lem:KMM94} there exists a $(K_X + \Delta)$-negative extremal ray $R_0$ such that 
$$(K_X + \Delta + \lambda_0 L)\cdot R_0=0;$$ 
 let $\varphi_0$ be the divisorial contraction or a flip associated to $R_0$. Then (b) and (c) are clear by the Cone theorem and by the Negativity lemma \cite[Theorem 3.7 and Lemma 3.39]{KM98}.

Assume the sequence is constructed for some $i\geq0$, and set
$$\lambda_{i+1}:=\sup\{t\geq0\mid t(K_{X_{i+1}}+\Delta_{i+1})+L_{i+1} \text{ is nef}\}.$$ 
Define $M_{i+1}:=\lambda_i(K_{X_{i+1}}+\Delta_{i+1})+L_{i+1}$. Then $M_{i+1}$ is nef by (b), and hence $\lambda_{i+1}\geq\lambda_i$. Set 
$$\mu_{i+1}:=\sup\{t\geq0\mid t(K_{X_{i+1}}+\Delta_{i+1})+M_{i+1} \text{ is nef}\}.$$
By the identity
$$t(K_{X_{i+1}}+\Delta_{i+1})+M_{i+1}=(t+\lambda_i)(K_{X_{i+1}}+\Delta_{i+1})+L_{i+1}$$
we have $\lambda_{i+1}=\mu_{i+1}+\lambda_i$.
If $K_{X_{i+1}}+\Delta_{i+1}$ is nef, our construction stops. Otherwise, $\mu_{i+1}$ is a non-negative rational number by Lemma \ref{lem:KMM94}, and there exists a $(K_{X_{i+1}}+\Delta_{i+1})$-negative extremal ray $R_{i+1}$ such that 
$$\big(\mu_{i+1}(K_{X_{i+1}}+\Delta_{i+1})+M_{i+1}\big)\cdot R_{i+1}=0.$$
Then we let $\varphi_{i+1}\colon X_{i+1}\to X_{i+2}$ be the divisorial contraction or the flip associated to $R_{i+1}$. 

\subsection{An auxiliary MMP result}

The following result is well-known.

\begin{lem}\label{lem:numtrivial}
Let $(X,\Delta)$ be a klt pair of dimension $n$, let $L$ be a nef Cartier divisor on $X$, and let $m>2n$ be a rational number. Let $C$ be a curve on $X$ such that 
$$(K_X+\Delta+mL)\cdot C=0.$$ 
Then 
$$(K_X+\Delta)\cdot C = L\cdot C=0.$$ 
\end{lem}

\begin{proof}
By \cite[Lemma 6.7]{Deb01} there are positive real numbers $\alpha_i$ and extremal classes $\gamma_i$ of the cone of effective curves $\NEb(X)$ such that 
$$\textstyle [C]=\sum\alpha_i \gamma_i.$$ 
It suffices to show that 
$$(K_X+\Delta)\cdot \gamma_i\geq0 $$
holds for all $i$: indeed, then $(K_X+\Delta)\cdot C\geq0$, which together with $(K_X+\Delta+mL)\cdot C=0$ implies the lemma since $L$ is nef.

To this end, assume that $(K_X+\Delta)\cdot \gamma_i<0$ for some $i$. Then by the boundedness of extremal rays, \cite[Theorem 1]{Kaw91}, there exists a curve $\Gamma_i$ whose class is proportional to $\gamma_i$ such that 
$$-2n\leq(K_X+\Delta)\cdot \Gamma_i<0.$$
Since $(K_X+\Delta+mL)\cdot \Gamma_i=0$, we have $L\cdot \Gamma_i>0$, and since $L$ is Cartier, this implies $L\cdot \Gamma_i\geq1$. Thus, 
$$(K_X+\Delta+mL)\cdot \Gamma_i>0,$$
a contradiction. This finishes the proof.
\end{proof}

\subsection{Nef reduction}

We need the notion of the nef reduction of a nef divisor. The following is the main result of \cite{BCE+}.

\begin{thm}\label{thm:nefreduction}
Let $L$ be a nef divisor on a normal projective variety $X$. Then there exists an almost holomorphic dominant rational map $f\colon X\dashrightarrow Y$ with connected fibres, called the \emph{nef reduction of $L$}, such that:
\begin{enumerate}
\item[(i)] $L$ is numerically trivial on all compact fibres $F$ of $f$ with $\dim F=\dim X-\dim Y$,
\item[(ii)] for every very general point $x\in X$ and every irreducible curve $C$ on $X$ passing through $x$ and not contracted by $f$, we have $L\cdot C>0$.
\end{enumerate}
The map $f$ is unique up to birational equivalence of $Y$ and the \emph{nef dimension} $n(X,L)$ is defined as the dimension of $Y$:
$$ n(X,L) = \dim Y. $$
\end{thm}

Note that in  \cite{BCE+}, ``general'' means ``very general", whereas general points are called ``generic". 
It is immediate from the definition that, with notation above, we have $n(X,L)=\dim X$ if and only if $L\cdot C>0$ for every irreducible curve $C$ on $X$ passing through a very general point on $X$. This yields:

\begin{lem}\label{lem:maxneffibres}
Let $f\colon X\to Y$ be a surjective morphism with connected fibres between normal projective varieties, and let $D$ be a nef $\Q$-divisor on $X$. If $n(X,D)=\dim X$, then for a general fibre $F$ of $f$ we have $n(F,D|_F)=\dim F$.
\end{lem}

\begin{proof}
Let $U$ be a very general subset of $X$ such that for every curve $C$ on $X$ passing through a point in $U$ we have $D\cdot C>0$. Then $U\cap F$ is a very general subset of $F$, and for every curve $C'$ on $F$ passing through a point in $U\cap F$ we have $D|_F\cdot C'>0$. This implies the lemma.
\end{proof}

The following easy result will be used often in the remainder of the paper.

\begin{lem}\label{lem:nefredsum}
Let $X$ be a normal projective variety of dimension $n$. Let $F$ and $G$ be pseudoeffective divisors on $X$ such that $F+G$ is nef and $n(X,F+G)=n$. Then for every positive real number $t$ and every curve $C$ passing through a very general point of $X$ we have $(F+tG)\cdot C>0$. In particular, if $F+tG$ is nef for some $t>0$, then $n(X,F+tG)=n$.
\end{lem}

\begin{proof}
By assumption, there is a very general subset $U$ of $X$ such that for every curve $C$ passing through a point in $U$ we have $(F+G)\cdot C>0$. Since $F$ and $G$ are pseudoeffective, there is a very general subset $U'$ of $X$ such that for every curve $C$ passing through a point in $U'$ we have $F\cdot C\geq0$ and $G\cdot C\geq0$. Now, every curve $C$ passing through a point in $U\cap U'$ satisfies the conclusion of the lemma.
\end{proof}

We also need the following property \cite[Proposition 2.1]{BP04}.

\begin{lem}\label{lem:nefdimensionandpullback}
Let $f\colon X\to Y$ be a proper surjective morphism of normal
projective varieties and let $L$ be a nef line bundle on $Y$. Then $n(X,f^*L) = n(Y,L)$.
\end{lem}

\subsection{Descent of num-effectivity and num-semiampleness}

We need two descent results for num-effectivity. We recall first the following well-known lemma.

\begin{lem}\label{lem:finite}
Let $f\colon X\to Y$ be a finite surjective morphism between varieties, and assume that $Y$ is normal. Then for every Weil divisor $D$ on $Y$ we have 
$$f_*f^*D=(\deg f)D.$$
\end{lem}

\begin{proof}
Let $X^\circ$ and $Y^\circ$ be the smooth loci of $X$ and $Y$, respectively, and consider the open sets $U_Y:=Y^\circ\setminus f(X\setminus X^\circ)\subseteq Y$ and $U_X:=f^{-1}(U_Y)\subseteq X$. By normality of $Y$, we have $\codim_X(X\setminus U_X)=\codim_Y(Y\setminus U_Y)\geq2$, and the map $f|_{U_X}\colon U_X\to U_Y$ is flat by \cite[Exercise III.9.3(a)]{Har77}. Then $(f|_{U_X})_*(f|_{U_X})^*(D|_{U_Y})=(\deg f)(D|_{U_Y})$ by \cite[Example 1.7.4]{Ful98}, and the lemma follows.
\end{proof}

\begin{lem}\label{lem:descenteff1}
Let $f\colon X\to Y$ be a birational morphism between two projective varieties with rational singularities. Let $D$ be a $\Q$-Cartier divisor on $Y$ such that there exists an effective (respectively semiample) $\Q$-divisor $G$ on $X$ with $f^*D\equiv G$. Then $f_*G$ is an effective (respectively semiample) $\Q$-Cartier divisor and $D\equiv f_*G$.
\end{lem}

\begin{proof}
By assumption, there exists a divisor $L_X\in\Pic^0(X)$ and a positive integer $m$ such that $mD$ and $mG$ are Cartier and
$$mf^*D+L_X=mG.$$
Since $X$ and $Y$ have rational singularities, the pullback map $f^*\colon \Pic^0(Y)\to\Pic^0(X)$ is an isomorphism, hence there exists $L_Y\in\Pic^0(Y)$ such that $L_X=f^*L$. Therefore,
$$mD+L_Y= f_*G,$$
which proves the lemma in the case of num-effectivity.

If $G$ is semiample, denote $G':=f_*G$. Then $G\equiv f^*D\equiv f^*G'$, hence $G=f^*G'$ by the Negativity lemma \cite[Lemma 3.39]{KM98}. But then $G'$ is clearly semiample.
\end{proof}

\begin{lem}\label{lem:descenteff2}
Let $f\colon X\to Y$ be a finite surjective morphism between two normal projective varieties, where $Y$ has rational singularities. Let $D$ be a $\Q$-Cartier divisor on $Y$ such that there exists an effective (respectively semiample) $\Q$-divisor $G$ on $X$ with $f^*D\equiv G$. Then $f_*G$ is an effective (respectively semiample) $\Q$-Cartier and $D\equiv \frac{1}{\deg f}f_*G$.
\end{lem}

\begin{proof}
Let $\pi\colon Y'\to Y$ be a desingularisation and let $X'$ be the normalisation of the main component of the fibre product $X\times_Y Y'$, so that we have the commutative diagram
\[
\xymatrix{ 
X' \ar[d]_{f'} \ar[r]^{\sigma} & X \ar[d]^{f}\\
Y'\ar[r]_{\pi} & Y,
}
\]
where $f'$ is finite and $\sigma$ is birational.  By assumption, we have 
$$f'^*\pi^*D=\sigma^*f^*D\equiv\sigma^*G,$$
and hence
\begin{equation}\label{eq:push}
f'_*f'^*\pi^*D\equiv f'_*\sigma^*G
\end{equation}
since $Y'$ is smooth. Then Lemma \ref{lem:finite} gives
$$(\deg f')\pi^*D\equiv f'_*\sigma^*G,$$
and thus, by Lemma \ref{lem:descenteff1},
$$D\equiv \frac{1}{\deg f'}\pi_*f'_*\sigma^*G.$$
Since $\pi_*f'_*\sigma^*G=f_*\sigma_*\sigma^*G=f_*G$ and $\deg f=\deg f'$, the lemma follows in the case of num-effectivity.

Finally, assume that $G$ is semiample, so that $gG$ is basepoint free for some positive integer $g$. For any point $y\in Y$, a general element $G'\in|gG|$ avoids the finite set $f^{-1}(y)$, hence $y\notin\Supp f_*G'$. Since $f_*G'\in |gf_*G|$, the divisor $gf_*G$ is basepoint free.
\end{proof}

\section{Birational descent of nef divisors}

If $f\colon X\to Y$ is a morphism between normal projective varieties and if $L$ is a $\Q$-divisor on $X$ which is numerically trivial on a general fibre of $f$, then $L$ does not have to be numerically equivalent to a pullback of a $\Q$-divisor from $Y$. 

However, in this section we show that $L$ is a pullback after a birational base change. This was shown when $L$ is additionally effective in \cite[Corollary III.5.9]{Nak04}, when $L$ is abundant in \cite[Proposition 2.1]{Kaw85}, or if it is numerically trivial on every fibre of $f$ in \cite[Theorem 1.2]{Leh15}. Lehmann in \cite[Theorem 1.3]{Leh15} gives a very satisfactory solution when $L$ is not nef, but only pseudoeffective. The following lemma, which is crucial for the rest of the paper, makes that result somewhat more precise.

\begin{lem}\label{lem:lehmann}
Let $f\colon X\to Z$ be a surjective morphism with connected fibres between normal projective varieties and let $L$ be a nef divisor on $X$ such that $L|_F\equiv 0$ for a general fibre $F$ of $f$. Then there exist a birational morphism $\pi\colon Z'\to Z$ from a smooth projective variety $Z'$, a $\Q$-divisor $D$ on $Z'$, and a commutative diagram
\[
\xymatrix{ 
X' \ar[d]_{f'} \ar[r]^{\pi'} & X \ar[d]^{f}\\
Z'\ar[r]_{\pi} & Z,
}
\]
where $X'$ is the normalisation of the main component of the fibre product $Z'\times_Z X$, such that $\pi'^*L\equiv f'^*D$. 
\end{lem}

\begin{proof}
By \cite[Theorem 1.3]{Leh15} there exists a birational morphism $\pi_1\colon Z_1\to Z$ from a smooth\footnote{The smoothness of $Z_1$ follows from the proof of \cite[Theorem 5.3]{Leh15}.} projective variety $Z_1$, a birational morphism $\pi_1'\colon X_1\to X$ from a smooth variety $X_1$, a morphism $f_1\colon X_1\to Z_1$ with connected fibres, and a pseudoeffective $\Q$-divisor $D_1$ on $Z_1$ such that $\pi_1\circ f_1=f\circ\pi_1'$ and such that 
$$\pi_1'^*L\equiv P_\sigma(f_1^*D_1).$$
In particular, $P_\sigma(f_1^*D_1)$ is nef. By Lemma \ref{lem:Nakayama3} there exists a birational morphism $\pi_2\colon Z'\to Z_1$ from a smooth variety $Z'$ such that $P_\sigma(\pi_2^*D_1)$ is nef. Set $\pi:=\pi_1\circ\pi_2$, let $X_2$ be a resolution of the main component of the fibre product $X_1\times_{Z_1} Z'$, and let $\theta\colon X_2\to X$, $\pi'_2\colon X_2\to X_1$ and $f_2\colon X_2\to Z'$ be the induced maps.
\[
\xymatrix{ 
X_2\ar[r]_{\pi_2'}\ar[d]_{f_2}\ar@/^1pc/[rr]^{\theta} & X_1 \ar[d]_{f_1} \ar[r]_{\pi_1'} & X \ar[d]^{f}\\
Z'\ar[r]_{\pi_2}\ar@/_1.3pc/[rr]_\pi & Z_1\ar[r]_{\pi_1} & Z
}
\]
Then we have $P_\sigma(\pi_2'^*f_1^*D_1)=\pi_2'^*P_\sigma(f_1^*D_1)$ and $P_\sigma(f_2^*\pi_2^*D_1)=f_2^*P_\sigma(\pi_2^*D_1)$ by Lemma \ref{lem:Nakayama2}, hence
\begin{align}
\theta^*L&=\pi_2'^*\pi_1'^*L\equiv \pi_2'^*P_\sigma(f_1^*D_1)=P_\sigma(\pi_2'^*f_1^*D_1)\label{eq:lehnak}\\
&=P_\sigma(f_2^*\pi_2^*D_1)=f_2^*P_\sigma(\pi_2^*D_1).\notag
\end{align}
We now set $D:=P_\sigma(\pi_2^*D_1)$, and let $X'$ be the normalisation of the main component of the fibre product $Z'\times_Z X$ with projections $f'\colon X'\to Z'$ and $\pi'\colon X'\to X$. 
Then there exists a birational morphism $\varphi\colon X_2\to X'$ and a commutative diagram
\[
\xymatrix{ 
X_2\ar[r]_{\varphi}\ar[dr]_{f_2}\ar@/^1pc/[rr]^{\theta} & X' \ar[d]_{f'} \ar[r]_{\pi'} & X \ar[d]^{f}\\
& Z'\ar[r]_{\pi} & Z.
}
\]
Then by \eqref{eq:lehnak} we have
$$\varphi^*(\pi'^*L-f'^*D)\equiv0,$$
hence $\pi'^*L\equiv f'^*D$ as desired.
\end{proof}

\section{Achieving nefness}\label{sec:reductions}

Let $(X,\Delta)$ be a projective klt pair with $K_X+\Delta$ pseudoeffective, and let $L$ be a nef Cartier divisor on $X$. In this section we show that we can run an $L$-trivial $(K_X+\Delta)$-MMP which, assuming the termination of flips, terminates with a model on which the strict transform of $K_X+\Delta+mL$ is nef for $m\gg0$. In particular, the strict transform of $L$ stays nef and Cartier. This result is the first step towards the proofs of our main results.

As a remarkable corollary, we obtain that the Generalised Nonvanishing Conjecture follows -- modulo the termination of flips -- from the nonvanishing part of the  Generalised Abundance Conjecture.

The following result is contained in \cite[Section 3]{BH14} and \cite[Section 4]{BZ16}, and we follows the proofs therein.

\begin{pro}\label{pro:contocon}
Assume the termination of klt flips in dimension $n$. Let $(X,\Delta)$ be a projective klt pair of dimension $n$ such that $K_X+\Delta$ is pseudoeffective, let $L$ be a nef Cartier divisor on $X$ and let $m>2n$ be a positive integer. Then there exists an $L$-trivial $(K_X+\Delta)$-MMP $\varphi\colon X\dashrightarrow Y$ such that $K_Y+\varphi_*\Delta+m\varphi_*L$ is nef.
\end{pro}

\begin{proof}
Let $A$ be an effective big divisor such that:
\begin{enumerate}
\item[(a)] $K_X+\Delta+mL+A$ is nef, and
\item[(b)] for all $0<s<1$ there exists an effective $\Q$-divisor $E_s$ such that $E_s\sim_\Q \Delta+mL+sA$ and the pair $\big(X,E_s+(1-s)A\big)$ is klt.
\end{enumerate}
Such a divisor exists: by Bertini's theorem we may take $A$ to be a sufficiently ample $\Q$-divisor with small coefficients. However, we want to stress that, for inductive purposes, it is necessary to assume only bigness of $A$.

Let 
$$\lambda:=\min\{t\geq0\mid K_X+\Delta+mL+tA\text{ is nef}\}\leq1.$$
If $\lambda=0$, there is nothing to prove, hence we may assume that $\lambda>0$. Pick a positive rational number $\mu<\lambda$ and let $E_\mu$ be a divisor as in (b). Then the pair $\big(X,E_\mu+(\lambda-\mu)A\big)$ is also klt, and note that
\begin{equation}\label{eq:ray}
K_X+E_\mu+(\lambda-\mu)A\sim_\Q K_X+\Delta+mL+\lambda A.
\end{equation}
Then by \cite[Lemma 2.2]{KMM94} there exists a $(K_X+E_\mu)$-negative extremal ray $R$ such that 
$$\big(K_X+E_\mu+(\lambda-\mu)A\big)\cdot R=0.$$ 
In particular, we have $A\cdot R>0$, hence \eqref{eq:ray} implies that
\begin{equation}\label{eq:negative}
(K_X+\Delta+mL)\cdot R<0,
\end{equation}
and therefore
\begin{equation}\label{eq:negative1}
(K_X+\Delta)\cdot R<0,
\end{equation}
since $L$ is nef. Let $c_R\colon X\to Y$ be the contraction of $R$. By the boundedness of extremal rays \cite[Theorem 1]{Kaw91}, there exists a curve $C$ whose class belongs to $R$ such that 
$$(K_X+\Delta)\cdot C\geq-2n.$$ 
If $L\cdot C>0$, then $L\cdot C\geq 1$ since $L$ is Cartier, hence $(K_X+\Delta+mL)\cdot R>0$, which contradicts \eqref{eq:negative}. Therefore, $L\cdot C=0$ and hence $c_R$ is $L$-trivial. 

If $c_R$ is divisorial, set $\theta := c_R$, and if $c_R$ is small, let $\theta\colon X\dashrightarrow X'$ be the corresponding flip. Then $\theta_*L$ is a nef Cartier divisor and $K_{X'}+\theta_*\Delta+m\theta_*L+\lambda\theta_*A$ is nef. Since the map $\theta$ is $\big(K_X+E_s+(\lambda-s)A\big)$-negative for each $0<s<\lambda$, for each $0<s\leq\lambda$ the pair $\big(X',\theta_*E_s+(\lambda-s)\theta_*A\big)$ is klt. Now we replace $X$ by $X'$, $\Delta$ by $\theta_*\Delta$, $L$ by $\theta_*L$, $A$ by $\lambda \theta_*A$, and $E_s$ by $\theta_*E_{s/\lambda}$. Then (a) and (b) continue to hold, and therefore we can continue the procedure. Since this process defines a sequence of operations of a $(K_X+\Delta)$-MMP by \eqref{eq:negative1}, the process must terminate by assumption.
\end{proof}

The following result implies Theorem \ref{thm:reduction}.

\begin{cor}\label{cor:genAbundancereduction1}
Assume the termination of klt flips in dimension $n$. Then the  nonvanishing part of the  Generalised Abundance Conjecture in dimension $n$ implies the Generalised Nonvanishing Conjecture in dimension $n$. \end{cor}

In other words, it is sufficient to prove the  Generalised Nonvanishing Conjecture under the additional assumption that $K_X + \Delta +L$ is nef, provided the termination of klt flips.

\begin{proof}
Let $(X,\Delta)$ be a projective klt pair of dimension $n$ such that $K_X+\Delta$ is pseudoeffective, let $L$ be a nef Cartier divisor on $X$, and let $m>2n$ be a positive integer. Proposition \ref{pro:contocon} implies that there exists an $L$-trivial $(K_X+\Delta)$-MMP $\varphi\colon X\dashrightarrow Y$ such that $K_Y+\varphi_*\Delta+m\varphi_*L$ is nef, hence $K_Y+\varphi_*\Delta+m\varphi_*L$ is nef for every $t\geq m$. In particular, there exists a resolution of indeterminacies $(p,q)\colon W\to X\times Y$ of $\varphi$ with $W$ smooth such that for all $t\geq0$ we have
$$p^*(K_X+\Delta+tL)\sim_\Q q^*(K_Y+\varphi_*\Delta+t\varphi_*L)+E,$$
where $E$ is effective and $q$-exceptional. By the nonvanishing part of the Generalised Abundance Conjecture, for each $t\geq m$ there exists an effective $\Q$-divisor $G_t\equiv K_Y+\varphi_*\Delta+t\varphi_*L$, hence
$$K_X+\Delta+tL\equiv p_*(q^*G_t+E).$$
Since $ p_*(q^*G_t+E)$  is $\Q$-effective, the divisor $K_X+\Delta+tL$ is num-effective for $t\geq m$. Since we assume the termination of flips and the nonvanishing part of the Generalised Abundance Conjecture, the divisor $K_X+\Delta$ num-effective, and this implies that the divisor $K_X+\Delta+tL$ is num-effective for all $t\geq 0$, which was to be shown.
\end{proof}

\section{Proofs of Theorems \ref{main_theorem1} and \ref{main_theorem2}}\label{sec:main}

In this section we prove the main results of this paper, Theorems \ref{main_theorem1} and \ref{main_theorem2}. Corollary \ref{cor:dim3} is then immediate. We conclude with a finite generation result, Corollary \ref{cor:fingen}.

\medskip

We start with two preparatory results. Let $(X,\Delta)$ be a projective klt pair such that $K_X+\Delta$ is pseudoeffective, and let $L$ be a nef Cartier divisor on $X$. Then we show that the maximal nef dimension of $K_X + \Delta +L $ forces the Kodaira dimension of $K_X+\Delta+L$ to be maximal, assuming the Abundance Conjecture and the 
Generalised Nonvanishing Conjecture.

\begin{lem}\label{lem:nred=n}
Assume the Abundance Conjecture in dimension $n$ and the Generalised Nonvanishing Conjecture in dimension $n$. 

Let $(X,\Delta)$ be a projective klt pair of dimension $n$ with $K_X+\Delta$ nef, and let $L$ be a nef divisor on $X$. Then $n(X,K_X+\Delta+L)=n$ if and only if $K_X+\Delta+L$ is big.
\end{lem}

\begin{proof}
One direction is obvious: if $D$ is any divisor on $X$ which is nef and big, then $n(X,D) = n$. 

Conversely, assume that $n(X,K_X+\Delta+L)=n$. By the Generalised Nonvanishing Conjecture, there exists an effective $\Q$-divisor $M$ such that $K_X+\Delta+L\equiv M$.  Choose a positive rational number $\varepsilon$ such that the pair $(X,\Delta+\varepsilon M)$ is klt. 
Since we assume the Abundance Conjecture in dimension $n$, the nef divisor $K_X+\Delta+\varepsilon M$ is semiample. Let $\varphi\colon X\to Y$ be the associated Iitaka fibration. Then there exists an ample $\Q$-divisor $A$ on $Y$ such that 
$$K_X+\Delta+\varepsilon M\sim_\Q\varphi^*A,$$ 
hence by Lemma \ref{lem:nefdimensionandpullback},
$$n(X,K_X+\Delta+\varepsilon M)=n(Y,A)=\dim Y.$$
On the other hand, since $K_X+\Delta+\varepsilon M\equiv(1+\varepsilon)(K_X+\Delta)+\varepsilon L$, we have 
$$n(X,K_X+\Delta+\varepsilon M)=n$$
by Lemma \ref{lem:nefredsum}. Therefore, $\dim Y=n$, the morphism $\varphi$ is birational and $K_X+\Delta+\varepsilon M$ is big. Since
$$\textstyle (1+\varepsilon)(K_X+\Delta+L)=K_X+\Delta+\varepsilon M+L,$$
the result follows.
\end{proof}

\begin{lem}\label{lem:maxnefdimBig}
Assume the termination of klt flips in dimension $n$, the Generalised Nonvanishing Conjecture in dimensions at most $n-1$, and the Abundance Conjecture in dimensions at most $n$. 

Let $(X,\Delta)$ be an $n$-dimensional projective klt pair such that $K_X+\Delta$ is pseudoeffective and let $L$ be a nef $\Q$-divisor on $X$. Assume that $K_X+\Delta+L$ is nef and that $n(X,K_X+\Delta+L)=n$. 
\begin{enumerate}
\item[(i)] If $\nu(X,K_X+\Delta)>0$, then $K_X+\Delta+tL$ is big for every $t>0$. 
\item[(ii)] Assume additionally the Semiampleness Conjecture in dimension $n$. If $\nu(X,K_X+\Delta)=0$, then $K_X+\Delta+tL$ is big for every $t>0$. 
\end{enumerate}
\end{lem}

\begin{proof}
\emph{Step 1.}
As explained in \S\ref{subsec:scaling}, we may construct a sequence of divisorial contractions or flips 
$$\varphi_i\colon (X_i,\Delta_i)\dashrightarrow (X_{i+1},\Delta_{i+1})$$ 
of a $(K_X+\Delta)$-MMP with $(X_0,\Delta_0):=(X,\Delta)$, a sequence of divisors 
$$L_i:=(\varphi_{i-1}\circ\dots\circ\varphi_0)_*L$$
on $X_i$, and a sequence of rational numbers $\lambda_i$ for $i\geq0$ with the following properties:
\begin{enumerate}
\item[(a)] the sequence $\{\lambda_i\}_{i\geq0}$ is non-decreasing,
\item[(b)] each $\lambda_i(K_{X_i}+\Delta_i)+L_i$ is nef and $\varphi_i$ is $\big(\lambda_i(K_{X_i}+\Delta_i)+L_i\big)$-trivial,
\item[(c)] the map $\varphi_i\circ\dots\circ\varphi_0$ is $\big(s(K_X+\Delta)+L)$-negative for $s>\lambda_i$.
\end{enumerate}
Additionally, we claim that
\begin{enumerate}
\item[(d)] $(K_{X_i}+\Delta_i+tL_i)\cdot C_i>0$ for every curve $C_i$ on $X_i$ passing through a very general point on $X_i$ and for every $t>0$.
\end{enumerate}
Indeed, this holds on $X_0$ by Lemma \ref{lem:nefredsum}. If it holds on $X_{k-1}$, then we have $n\big(X_{k-1},\lambda_{k-1}(K_{X_{k-1}}+\Delta_{k-1})+L_{k-1}\big)=n$, and hence $n\big(X_k,\lambda_{k-1}(K_{X_k}+\Delta_k)+L_k\big)=n$ by (b) and by Lemma \ref{lem:nefdimensionandpullback}. The claim follows from Lemma \ref{lem:nefredsum}.

\medskip

\emph{Step 2.}
By the termination of flips in dimension $n$, this process must terminate with a pair $(X_\ell,\Delta_\ell)$ such that $K_{X_\ell}+\Delta_\ell$ is nef, and then (b) implies that
$$s(K_{X_\ell}+\Delta_\ell)+L_\ell\quad\text{is nef for all }s\geq\lambda_{\ell-1}.$$
By (d) we have 
$$n\big(X_\ell,s(K_{X_\ell}+\Delta_\ell)+L_\ell\big)=n\quad\text{for all }s\geq\lambda_{\ell-1}.$$
Then it suffices to show that $(s_0+1)(K_{X_\ell}+\Delta_\ell)+L_\ell$ is big for some fixed $s_0\geq\lambda_{\ell-1}$. Indeed, then the divisor $(s_0+1)(K_X+\Delta)+L$ is big thanks to (c). Since $K_X+\Delta$ and $L$ are pseudoeffective, one deduces immediately that $K_X+\Delta+tL$ is big for all $t>0$. 

Therefore, by replacing $X$ by $X_\ell$, $\Delta$ by $\Delta_\ell$ and $L$ by $s_0(K_{X_\ell}+\Delta_\ell)+L_\ell$, we may assume that $K_X+\Delta$ is nef, and it suffices to show that $K_X+\Delta+L$ is big.

\medskip

\emph{Step 3.}
Since we are assuming the Abundance Conjecture in dimension $n$, the divisor $K_X+\Delta$ is semiample and we have $\nu(X,K_X+\Delta)=\kappa(X,K_X+\Delta)$. Assume first that 
$$\nu(X,K_X+\Delta)=\kappa(X,K_X+\Delta)=0.$$
Then $K_X+\Delta\sim_\Q0$, and the lemma follows since in this case we assume the Semiampleness Conjecture in dimension $n$.

\medskip

\emph{Step 4.}
Therefore, we may assume that 
$$\nu(X,K_X+\Delta)=\kappa(X,K_X+\Delta)>0,$$
and let $\varphi\colon X\to Y$ be the Iitaka fibration associated to $K_X+\Delta$. Then for a general fibre $F$ of $\varphi$ we have 
$$n\big(F,(K_X+\Delta+L)|_F\big)=\dim F$$
by Lemma \ref{lem:maxneffibres}. This, together with Lemma \ref{lem:nred=n} and the assumptions of the lemma imply that $(K_X+\Delta+L)|_F$ is big, and hence $K_X+\Delta+L$ is $\varphi$-big. 

There exists an ample $\Q$-divisor $H$ on $Y$ such that $K_X+\Delta\sim_\Q\varphi^*H$. By \cite[Lemma 3.2.1]{BCHM} there exists an effective divisor $G$ on $X$ and a $\Q$-Cartier divisor $D$ on $Y$ such that
$$K_X+\Delta+L\sim_\Q G+\varphi^*D.$$
Note that $G$ is $\varphi$-nef and $\varphi$-big since $K_X+\Delta+L$ is nef and $\varphi$-big.

Let $\delta$ be a positive rational number such that the pair $(X,\Delta+\delta G)$ is klt. Then $K_X+\Delta+\delta G$ is $\varphi$-nef and $\varphi$-big, hence $K_X+\Delta+\delta G$ is $\varphi$-semiample 
by the relative Basepoint free theorem \cite[Theorem 3.24]{KM98}. Let 
$$\theta\colon X\to T$$
be the associated relative Iitaka fibration, and note that $\theta$ is birational. Let $\zeta\colon T\to Y$ be the induced morphism making the following diagram commutative:
\[
\xymatrix{ 
X \ar[d]_{\varphi} \ar[r]^{\theta} & T \ar[dl]^\zeta\\
\ \  Y. &
}
\]
Then there exists a $\zeta$-ample $\Q$-divisor $A$ on $T$ such that 
$$K_X+\Delta+\delta G\sim_\Q\theta^*A.$$
Since $A+\zeta^*(\delta D)$ is $\zeta$-ample, there exists a positive rational number $\lambda$ such that $A+\zeta^*(\delta D)+\zeta^*(\lambda H)$ is ample. Then
\begin{align*}
(1+\lambda)(K_X+\Delta)+\delta(K_X+\Delta+L)&\sim_\Q(K_X+\Delta+\delta G)+\varphi^*(\lambda H+\delta D)\\
&\sim_\Q\theta^*\big(A+\zeta^*(\lambda H+\delta D)\big),
\end{align*}
and therefore the divisor $K_X+\Delta+\frac{\delta}{1+\lambda+\delta} L$ is big. But then it is clear that $K_X+\Delta+L$ is big since $K_X+\Delta$ and $L$ are pseudoeffective. This finishes the proof.
\end{proof}

The following theorem is the main technical result of this paper.

\begin{thm}\label{thm:semiampleP_sigma}
Assume the termination of klt flips in dimension $n$, the Abundance Conjecture in dimensions at most $n$, the Semiampleness Conjecture in dimensions at most $n$, and the Generalised Nonvanishing Conjecture in dimensions at most $n-1$. 

Let $(X,\Delta)$ be an $n$-dimensional projective klt pair such that $K_X+\Delta$ is pseudoeffective and let $L$ be a nef Cartier divisor on $X$. Then there exists a birational morphism $w\colon W\to X$ from a smooth projective variety $W$ such that $P_\sigma\big(w^*(K_X+\Delta+L)\big)$ is a num-semiample $\Q$-divisor.
\end{thm}

\begin{proof}
\emph{Step 1.}
Fix any integer $m>2n$. In this step we show that we may assume the following:

\medskip

\emph{Assumption 1.} The divisor $K_X+\Delta+mL$ is nef.

\medskip

To this end, by Proposition \ref{pro:contocon} there exists a $(K_X+\Delta)$-MMP $\varphi\colon X\dashrightarrow Y$ which is $L$-trivial such that $K_Y+\varphi_*\Delta+m\varphi_*L$ is nef. Denote $\Delta_Y=\varphi_*\Delta$ and $L_Y:=\varphi_*L$, and note that $K_Y+\Delta_Y$ is pseudoeffective. 

Assume that there exists a birational morphism $v\colon V\to Y$ from a smooth projective variety $V$ such that $P_\sigma\big(v^*(K_Y+\Delta_Y+L_Y)\big)$ is num-semiample. Let $(w_1,w_V)\colon W_1\to X\times V$ be a resolution of indeterminacies of the birational map $v^{-1}\circ\varphi\colon X\dashrightarrow V$ such that $W_1$ is smooth:
\[
\xymatrix{ 
W_1 \ar[d]_{w_1} \ar[r]^{w_V} & V \ar[d]^{v}\\
X\ar@{-->}[r]_{\varphi} & Y.
}
\]
There exists an effective $(v\circ w_V)$-exceptional $\Q$-divisor $E_1$ on $W_1$ such that
$$w_1^*(K_X+\Delta)\sim_\Q w_V^*v^*(K_Y+\Delta_Y)+E_1,$$
and hence
$$w_1^*(K_X+\Delta+L)\sim_\Q w_V^*v^*(K_Y+\Delta_Y+L_Y)+E_1.$$
Thus,
\begin{align*}
P_\sigma\big(w_1^*(K_X+\Delta+L)\big)&\sim_\Q P_\sigma\big(w_V^*v^*(K_Y+\Delta_Y+L_Y)\big)\\
&=w_V^*P_\sigma\big(v^*(K_Y+\Delta_Y+L_Y)\big)
\end{align*}
by Lemmas \ref{lem:Nakayama1} and \ref{lem:Nakayama2}, hence $P_\sigma\big(w_1^*(K_X+\Delta+L)\big)$ is num-semiample.

Therefore, by replacing $X$ by $Y$, $\Delta$ by $\Delta_Y$, and $L$ by $L_Y$, we achieve Assumption 1.

\medskip

\emph{Step 2.}
Assume first that $n(X,K_X+\Delta+mL)=n$. Then $K_X+\Delta+mL$ is big by Lemma \ref{lem:maxnefdimBig}, hence $K_X+\Delta+L$ is also big. By Kodaira's trick, we may choose an effective $\Q$-divisor $E$ and an ample $\Q$-divisor $A$ on $X$ such that $K_X+\Delta+L\sim_\Q A+E$. Pick a small positive rational number $\varepsilon$ such that the pair $(X,\Delta+\varepsilon E)$ is klt. Then
\begin{equation}\label{eq:008b}
(1+\varepsilon)(K_X+\Delta+L)\sim_\Q K_X+(\Delta+\varepsilon E)+(L+\varepsilon A),
\end{equation}
and note that $L+\varepsilon A$ is ample. Denote $\Delta_X:=(\Delta+\varepsilon E)+(L+\varepsilon A)$. By \cite{BCHM}, we may run a $(K_X+\Delta_X)$-MMP $\beta\colon X\dashrightarrow B$ so that $K_B+\Delta_B$ is semiample, where $\Delta_B=\beta_*\Delta_X$. Let $(w_2,w_B)\colon W_2\to X\times B$ be a resolution of indeterminacies of $\beta$. There exists an effective $w_B$-exceptional $\Q$-divisor $E_2$ on $W_2$ such that
$$w_2^*(K_X+\Delta_X)\sim_\Q w_B^*(K_B+\Delta_B)+E_2,$$
and hence by \eqref{eq:008b} and by Lemma \ref{lem:Nakayama1} we have
\begin{align*}
(1+\varepsilon)P_\sigma\big(w_2^*(K_X+\Delta+L)\big)&\sim_\Q P_\sigma\big(w_2^*(K_X+\Delta_X)\big)\\
&\sim_\Q P_\sigma\big(w_B^*(K_B+\Delta_B)\big)=w_B^*(K_B+\Delta_B),
\end{align*}
hence $P_\sigma\big(w_2^*(K_X+\Delta+L)\big)$ is num-semiample. This finishes the proof when $n(X,K_X+\Delta+mL)=n$.

\medskip

\emph{Step 3.}
From now on we assume that $n(X,K_X+\Delta+mL)<n$. In this step we show that we may assume the following:

\medskip

\emph{Assumption 2.} There exists a morphism $\zeta\colon X\to Z$ to a smooth projective variety $Z$ with $\dim Z<\dim X$ and a nef $\Q$-divisor $L_Z$ on $Z$ such that $L\equiv \zeta^*L_Z$ and $\nu\big(F,(K_X+\Delta)|_F\big)=0$ for a general fibre $F$ of $\zeta$. {\it However, we may not any more assume that $K_X+\Delta+L$ is nef.}

\medskip

To this end, let $\zeta\colon X\dashrightarrow Z$ be the nef reduction of $K_X+\Delta+mL$, and recall that $\zeta$ is almost holomorphic. Then $K_X+\Delta+mL$ is numerically trivial on a general fibre of $\zeta$, hence both $K_X+\Delta$ and $L$ are numerically trivial on a general fibre of $\zeta$ by Lemma \ref{lem:numtrivial}.

Let $\big(\widetilde\xi,\widetilde\zeta\big)\colon\widetilde{X}\to X\times Z$ be a smooth resolution of indeterminacies of the map $\zeta$.
\[
\xymatrix{ 
\widetilde X \ar[d]_{\widetilde\xi} \ar[dr]^{\widetilde\zeta} & \\
X\ar@{-->}[r]_{\zeta} & Z,
}
\]
Write 
$$K_{\widetilde{X}}+\widetilde{\Delta}\sim_\Q\widetilde{\xi}^*(K_X+\Delta)+\widetilde{E},$$ 
where $\widetilde{\Delta}$ and $\widetilde E$ are effective $\Q$-divisors without common components; note that $K_{\widetilde{X}}+\widetilde{\Delta}$ is pseudoeffective. Denote $\widetilde{L}:=\widetilde{\xi}^*L$. This implies
\begin{equation}\label{eq:009a}
K_{\widetilde{X}}+\widetilde{\Delta}+\widetilde{L}\sim_\Q\widetilde{\xi}^*(K_X+\Delta+L)+\widetilde{E}.
\end{equation}
Furthermore, if $\widetilde F$ is a general fibre of $\widetilde \zeta$, then $\nu\big(\widetilde F,(K_{\widetilde{X}}+\widetilde{\Delta})|_{\widetilde{F}}\big)=0$ by Lemma \ref{lem:1}, and $\widetilde{L}$ is numerically trivial on $\widetilde{F}$. 

By Lemma \ref{lem:lehmann} there exists a birational morphism $\xi_Z\colon \widehat Z\to Z$ from a smooth projective variety $\widehat Z$, a nef $\Q$-divisor $L_{\widehat Z}$ on $\widehat Z$, a smooth projective variety $\widehat X$ and a commutative diagram
\[
\xymatrix{ 
\widehat X \ar[d]_{\widehat\xi} \ar[r]^{\widehat\zeta} & \widehat Z \ar[d]^{\xi_Z}\\
\widetilde X\ar[r]_{\widetilde \zeta} & Z
}
\]
such that, if we denote $\widehat L:=\widehat\xi^*\widetilde L$, then
$$\widehat L\equiv \widehat \zeta^*L_{\widehat Z}.$$

We may write
$$K_{\widehat X}+\widehat\Delta\sim_\Q\widehat\xi^*\big(K_{\widetilde{X}}+\widetilde \Delta\big)+\widehat E,$$
where $\widehat\Delta$ and $\widehat E$ are effective $\Q$-divisors on $\widehat X$ without common components; note that $K_{\widehat{X}}+\widehat{\Delta}$ is pseudoeffective. This last relation above implies
\begin{equation}\label{eq:009b}
K_{\widehat X}+\widehat\Delta+\widehat{L}\sim_\Q \widehat\xi^*\big(K_{\widetilde{X}}+\widetilde \Delta+\widetilde L\big)+\widehat E.
\end{equation}
Furthermore, if $\widehat F$ is a general fibre of $\widehat\zeta$, then $\nu\big(\widehat F,(K_{\widehat X}+\widehat\Delta)|_{\widehat F}\big)=0$ by Lemma \ref{lem:1}. 

Assume that there exists there exists a birational morphism $\widehat w\colon W_3\to \widehat X$ from a smooth projective variety $W_3$ such that $P_\sigma\big(\widehat w^*\big(K_{\widehat X}+\widehat\Delta+\widehat{L}\big)\big)$ is num-semiample. Set $w_3:=\widetilde\xi\circ\widehat{\xi}\circ\widehat{w}\colon W_3\to X$. By \eqref{eq:009a} and \eqref{eq:009b} we have
$$\widehat w^*\big(K_{\widehat X}+\widehat\Delta+\widehat{L}\big)\sim_\Q w_3^*(K_X+\Delta+L)+\widehat{w}^*\widehat{\xi}^*\widetilde E+\widehat{w}^*\widehat{E},$$
and since the effective $\Q$-divisor $\widehat{w}^*\widehat{\xi}^*\widetilde E+\widehat{w}^*\widehat{E}$ is $w_3$-exceptional, by Lemma \ref{lem:Nakayama1} we obtain
$$P_\sigma\big(\widehat w^*\big(K_{\widehat X}+\widehat\Delta+\widehat{L}\big)\big)\sim_\Q P_\sigma\big(w_3^*(K_X+\Delta+L)\big).$$
Thus, $P_\sigma\big(w_3^*(K_X+\Delta+L)\big)$ is num-semiample.

Therefore, by replacing $X$ by $\widehat X$, $\Delta$ by $\widehat\Delta$, $L$ by $\widehat L$, and $\zeta$ by $\widehat{\zeta}$, we achieve Assumption 2.

\medskip

\emph{Step 4.}
In this step we show that we may assume the following, which strengthens Assumption 2:

\medskip

\emph{Assumption 3.} There exists a morphism $\zeta\colon X\to Z$ to a smooth projective variety $Z$ with $\dim Z<\dim X$ and a nef $\Q$-divisor $L_Z$ on $Z$ such that $L\equiv \zeta^*L_Z$, the divisor $K_X+\Delta$ is $\zeta$-semiample, and $(K_X+\Delta)|_F\sim_\Q0$ for a general fibre $F$ of $\zeta$.

\medskip

To this end, since $\nu\big(X,(K_X+\Delta)|_F\big)=0$ for a general fibre $F$ of $\zeta$ by Assumption 2, the pair $(F,\Delta|_F)$ has a good model by \cite[Corollaire 3.4]{Dru11} and \cite[Corollary V.4.9]{Nak04}. Therefore, by \cite[Theorem 2.12]{HX13} we may run any relative $(K_X+\Delta)$-MMP $\theta\colon X\dashrightarrow X_{\min}$ with scaling of an ample divisor over $Z$, which terminates with a relative good model $X_{\min}$ of $(X,\Delta)$ over $Z$. Note that this MMP is $L$-trivial by Assumption 2. Let $\zeta_{\min}\colon X_{\min}\to Z$ be the induced morphism. Denote $\Delta_{\min}=\theta_*\Delta$ and $L_{\min}:=\theta_*L$, and note that $K_{X_{\min}}+\Delta_{\min}$ is pseudoeffective.
\[
\xymatrix{ 
X \ar@{-->}[rr]^\theta \ar[rd]_\zeta & & X_{\min} \ar[dl]^{\zeta_{\min}}\\
& Z & 
}
\]
Assume that there exists a birational morphism $v_{\min}\colon V_{\min}\to Y$ from a smooth projective variety $V_{\min}$ such that $P_\sigma\big(v_{\min}^*(K_{X_{\min}}+\Delta_{\min}+L_{\min})\big)$ is num-semiample. Then analogously as in Step 1 we show that there exists a birational morphism $w_4\colon W_4\to X$ from a smooth projective variety $W_4$ such that $P_\sigma\big(w_4^*(K_X+\Delta+L)\big)$ is num-semiample.

Therefore, by replacing $X$ by $X_{\min}$, $\Delta$ by $\Delta_{\min}$, and $L$ by $L_{\min}$, we achieve Assumption 3.

\medskip

\emph{Step 5.} 
We are now in the position to finish the proof of Theorem \ref{thm:semiampleP_sigma}. Let $\tau\colon X\to T$ be the relative Iitaka fibration over $Z$ associated to $K_X+\Delta$. Then $\dim T<n$ since $(K_X+\Delta)|_F\sim_\Q0$ by Assumption 3. 

If $\zeta_T\colon T\to Z$ is the induced morphism, there exists a $\zeta_T$-ample $\Q$-divisor $A$ on $T$ such that $K_X+\Delta\sim_\Q\tau^*A$. 
\[
\xymatrix{ 
X \ar[d]_\zeta \ar[r]^\tau & T\ar[dl]^(0.35){\zeta_T} \\
Z & 
}
\]
By \cite[Theorem 0.2]{Amb05a} there exists an effective $\Q$-divisor $\Delta_T$ on $T$ such that the pair $(T,\Delta_T)$ is klt and such that  $$K_X+\Delta\sim_\Q\tau^*(K_T+\Delta_T),$$
in particular, $K_T+\Delta_T$ is pseudoeffective. If we denote $L_T:=\zeta_T^*L_Z$, then by Assumption 3 we have
\begin{equation}\label{eq:induction}
K_X+\Delta+L\equiv \tau^*(K_T+\Delta_T+L_T).
\end{equation}
By induction on the dimension, there exists a smooth projective variety $W_T$ and a birational morphism $w_T\colon W_T\to T$ such that $P_\sigma\big(w_T^*(K_T+\Delta_T+L_T)\big)$ is num-semiample. 

Let $W$ a desingularisation of the main component of the fibre product $X\times_T W_T$ and consider the induced commutative diagram
\begin{equation}\label{dia1}
\begin{gathered} 
\xymatrix{ 
W \ar[r]^{\tau_W}\ar[d]_w & W_T \ar[d]^{w_T} \\
X \ar[r]_\tau & T.
}
\end{gathered}
\end{equation}
Then
\begin{align*}
P_\sigma\big(w^*(K_X+\Delta+L)\big)&\equiv P_\sigma\big(w^*\tau^*(K_T+\Delta_T+L_T)\big) & \text{by \eqref{eq:induction}}\\
&=P_\sigma\big(\tau_W^*w_T^*(K_T+\Delta_T+L_T)\big) & \text{by \eqref{dia1}}\\
&=\tau_W^*P_\sigma\big(w_T^*(K_T+\Delta_T+L_T)\big) & \text{by Lemma \ref{lem:Nakayama2}}
\end{align*}
and therefore, $P_\sigma\big(w^*(K_X+\Delta+L)\big)$ is num-semiample. This finishes the proof. 
\end{proof}

\begin{rem} \label{rem:im} 
Notice that the Semiampleness Conjecture was used only in Step 2 of the proof of Theorem \ref{thm:semiampleP_sigma}, and only in the case $\nu(X,K_X+\Delta)=0$. Therefore, if $\nu(X,K_X + \Delta) > 0$ in Theorem  \ref{thm:semiampleP_sigma}, we only need to assume the termination of klt flips in dimension $n$, the Abundance Conjecture in dimension at most $n$, and the Generalised Nonvanishing Conjecture in dimension at most $n-1$.
\end{rem}

We are now able to deduce our main results immediately.

\begin{proof}[Proof of Theorem \ref{main_theorem1}]
By induction on the dimension, the Generalised Nonvanishing Conjecture in dimension $n-1$ is implied by the termination of klt flips in dimensions at most $n-1$, the Abundance Conjecture in dimensions at most $n-1$ and additionally  the Semiampleness Conjecture in dimensions at most $n-1$ in the case of Calabi-Yau pairs, see Remark \ref{rem:im}.

Therefore, Theorem \ref{thm:semiampleP_sigma} and the assumptions of Theorem \ref{main_theorem1} show that there exists a birational morphism $w\colon W\to X$ from a smooth projective variety $W$ such that $P_\sigma\big(w^*(K_X+\Delta+L)\big)$ is a num-semiample $\Q$-divisor. Since $N_\sigma\big(w^*(K_X+\Delta+L)\big)\geq0$, the divisor $w^*(K_X+\Delta+L)$ is num-effective, and hence $K_X+\Delta+L$ is num-effective by Lemma \ref{lem:descenteff1}. 
\end{proof}

\begin{proof}[Proof of Theorem \ref{main_theorem2}]
As in the previous proof, Theorem \ref{thm:semiampleP_sigma}, Remark \ref{rem:im} and the assumptions of Theorem \ref{main_theorem2} show that there exists a birational morphism $w\colon W\to X$ from a smooth projective variety $W$ 
such that $$P_\sigma\big(w^*(K_X+\Delta+L)\big)$$ is a num-semiample $\Q$-divisor. Since $K_X+\Delta+L$ is nef, we have $$P_\sigma\big(w^*(K_X+\Delta+L)\big)=w^*(K_X+\Delta+L),$$ hence $K_X+\Delta+L$ is num-semiample by Lemma \ref{lem:descenteff1}.
\end{proof}

Note that Corollary \ref{cor:dim3} is an immediate consequence of Theorems \ref{main_theorem1} and \ref{main_theorem2}, since the assumptions in Theorems \ref{main_theorem1} and \ref{main_theorem2} are 
satisfied in dimension $3$. 

We conclude this section with the following pleasant corollary of Theorem \ref{thm:semiampleP_sigma}.

\begin{cor}\label{cor:fingen}
Assume the termination of klt flips in dimensions at most $n$, the Abundance Conjecture in dimensions at most $n$, and the Semiampleness Conjecture in dimensions at most $n$. 

Let $(X,\Delta)$ be an $n$-dimensional projective klt pair such that $K_X+\Delta$ is pseudoeffective and let $L$ be a nef Cartier divisor on $X$. Then there exists an effective $\Q$-divisor $D\equiv K_X+\Delta+L$ such that the section ring 
$$R(X,D)=\bigoplus_{k\in\N}H^0(X,kD)$$
is finitely generated. 
\end{cor}

\begin{proof}
As in the proof of Theorem \ref{main_theorem1}, the assumptions of Corollary \ref{cor:fingen} and Theorem \ref{thm:semiampleP_sigma} show that there exists a birational morphism $w\colon W\to X$ from a smooth projective variety $W$ and a semiample $\Q$-divisor 
$$P_W\equiv P_\sigma\big(w^*(K_X+\Delta+L)\big).$$
Denote $D_W:=P_W+N_\sigma\big(w^*(K_X+\Delta+L)\big)$. Then $D_W\equiv w^*(K_X+\Delta+L)$ and by the proof of \cite[Theorem 5.5]{Bou04} we have
$$R(W,D_W)\simeq R(W,P_W).$$
Therefore, $R(W,D_W)$ is finitely generated. Denote
$$D:=w_*D_W.$$
Then $D$ is a $\Q$-Cartier $\Q$-divisor such that $D\equiv K_X+\Delta+L$ by Lemma \ref{lem:descenteff1}, and since $w^*D\equiv D_W$, we obtain $w^*D=D_W$ by the Negativity lemma \cite[Lemma 3.39]{KM98}. But then the ring $R(X,D)\simeq R(W,D_W)$ is finitely generated.
\end{proof}

\section{Counterexamples and applications}\label{AppB}

\subsection{Counterexamples}

In this subsection we collect several examples to demonstrate that several of the assumptions in the Generalised Nonvanishing and the Generalised Abundance Conjectures cannot be removed.

\begin{exa}
The following classical example demonstrates that the Generalised Abundance Conjecture does not hold if one assumes that the pair $(X,\Delta)$ is log canonical but not klt.

Let $C \subseteq \PS^2$ be a smooth elliptic curve and pick not necessarily distinct points $P_1, \ldots, P_9 \in C$. Let $\pi\colon X \to \PS^2$ be the blow up of these $9$ points. Then ${-}K_X$ is nef with $K_X^2 = 0$ and the strict transform $C'$ of $C$ in $X$ belongs to the linear system $|{-}K_X|$. Hence, $(X,C')$ is log canonical. It is known that for certain choices of points $P_1, \ldots, P_9$, we have $h^0(X,{-}mK_X)=1$ for all $m\geq0$, and therefore ${-}K_X$ is not semiample; see \cite[Example 2]{DPS96} for details. Setting $\Delta:=C'$ and $L := {-}K_X$ shows that the Generalised Abundance Conjecture fails in the log canonical setting.
\end{exa}

It remains open whether the Generalised Nonvanishing Conjecture holds for log canonical pairs. Furthermore, it is unclear whether the assumption of pseudoeffectivity of $K_X+\Delta$ can be removed from the Generalised Nonvanishing Conjecture, cf.\ \cite[Question 3.5]{BH14}. The following example shows, however, that the pseudoeffectivity is necessary in the formulation of the Generalised Abundance Conjecture.

\begin{exa}
This example is \cite[Example 1.1]{Sho00}, see also \cite[\S2.5]{DPS01}. It shows that the Generalised Abundance Conjecture fails if $K_X+\Delta$ is not pseudoeffective.

Let $E$ be a smooth elliptic curve and let $\mathcal E$ a rank $2$ vector bundle on $E$ which is the unique non-split extension
$$0\to\OO_E\to\mathcal{E}\to\OO_E\to0.$$
Set $X:=\PS(\mathcal{E})$ and let $\pi\colon X\to E$ be the projection. The surface $X$ contains a unique section $C$ of $\pi$ such that $C^2=0$, $C$ is nef, and $K_X+2C\sim0$. Moreover, $C$ is the unique effective divisor in the numerical class of $C$. In particular, if we set $\Delta:=\frac12C$ and $L:=2C$, then $(X,\Delta)$ is klt, $K_X+\Delta$ is not pseudoeffective, $K_X+\Delta+L$ is nef but not num-semiample.
\end{exa}

\begin{rem}
Let $(X,\Delta)$ be a klt pair, where $\Delta$ is an $\R$-divisor such that $K_X+\Delta$ is pseudoeffective and $\nu(X,K_X+\Delta)\geq1$. If $L$ is a nef $\R$-divisor, then (assuming the conjectures of the MMP) the proofs in this paper show that $K_X+\Delta+L$ is num-effective; and similarly if $K_X+\Delta+L$ is additionally nef, then $K_X+\Delta+L$ is num-semiample. 

This is, however, false if $\nu(X,K_X+\Delta)=0$. Indeed, \cite{FT17} gives an example of a projective K3 surface $X$ and a nef $\R$-divisor $L$ on $X$ such that $L$ is not semipositive, hence not semiample (in a suitable sense).
\end{rem}

\subsection{Generalised pairs}\label{subsec:genpairs}

We make a remark on generalised polarised pairs in the sense of \cite{BZ16}. Recall that a \emph{generalised (polarised) pair} consists of a pair $(X,\Delta)$, an $\R$-divisor $L$ on $X$, a log resolution $f\colon X'\to X$ of $(X,\Delta)$ and a nef $\R$-divisor $L'$ such that $L=f_*L'$. We also write the generalised pair as $(X,\Delta+L)$. If we write $K_{X'}+\Delta'+L'\sim_\Q f^*(K_X+\Delta+L)$, and if all coefficients of $\Delta'$ are smaller than $1$, then the generalised pair $(X,\Delta+L)$ is said to have klt singularities.

One can then formulate the Generalised Nonvanishing Conjecture for generalised polarised pairs: if $(X,\Delta+L)$ is a klt generalised polarised pair such that $K_X+\Delta$ is pseudoeffective, then $K_X+\Delta+L$ is num-effective. However, this is a trivial consequence of the Generalised Nonvanishing Conjecture from the introduction. Indeed, in the notation above, first note that by the Negativity lemma \cite[Lemma 3.39]{KM98} we may write $f^*L=L'+E$, where $E$ is effective and $f$-exceptional. We may also write $\Delta'=\Delta^+-\Delta^-$, where $\Delta^+$ and $\Delta^-$ are effective and have no common components. Since $f_*\Delta'=\Delta\geq0$, the divisor $\Delta^-$ is $f$-exceptional. Thus
$$K_{X'}+\Delta^+\sim_\Q f^*(K_X+\Delta)+\Delta^-+E,$$
hence $K_{X'}+\Delta^+$ is pseudoeffective and $(X',\Delta^+)$ is a klt pair. Then the Generalised Nonvanishing Conjecture implies that $K_{X'}+\Delta^++L'$ is num-effective. Since $f_*(K_{X'}+\Delta^++L')=K_X+\Delta+L$, the claim follows.

\subsection{Serrano's Conjecture}

Recall that a Cartier divisor $L$ on a projective variety $X$ is \emph{strictly nef} if $L\cdot C>0$ for every curve $C$ on $X$. The importance of strictly nef divisors stems from the following important conjecture of Serrano, see \cite{Ser95} and \cite[Conjecture 0.1]{CCP08}.

\begin{con}\label{con:serrano}
Let $X$ be a projective manifold of dimension $n$ and let $L$ be a strictly nef Cartier divisor on $X$. Then $K_X+tL$ is ample for every real number $t>n+1$.
\end{con}

The following result shows that this conjecture follows from the MMP and the Generalised Abundance Conjecture when $X$ is not uniruled.

\begin{lem}
Assume the termination of klt flips in dimensions at most $n$ and the Generalised Abundance Conjecture in dimension at most $n$. Let $X$ be a non-uniruled projective manifold of dimension $n$ and let $L$ be a strictly nef Cartier divisor on $X$. Then $K_X+tL$ is ample for every real number $t>n+1$.
\end{lem}

\begin{proof}
Fix a real number $t>n+1$. We claim first that the divisor $K_X+tL$ is strictly nef, and in particular, $n(X,K_X+tL)=n$ for any real number $t>n+1$. 

Indeed, let $C$ be any curve on $X$. Then there exist finitely many extremal classes $\gamma_i$ in the cone $\NEb(X)$ and positive real numbers $\lambda_i$ such that $[C]=\sum\lambda_i\gamma_i$, and it suffices to show that $(K_X+tL)\cdot \gamma_i>0$ for all $i$. If $K_X\cdot\gamma_i\geq0$, then immediately $(K_X+tL)\cdot \gamma_i>0$. If $K_X\cdot\gamma_i<0$, then by the boundedness of extremal rays \cite[Theorem 1.5]{Mor82} there is a curve $C_i$ on $X$ whose class is proportional to $\gamma_i$ such that $0<-K_X\cdot C_i\leq n+1$. Since $L$ is a Cartier divisor, we have $L\cdot C_i\geq1$, and hence $(K_X+tL)\cdot C_i>0$. This implies the claim.

The divisor $K_X$ is pseudoeffective by \cite[Corollary 0.3]{BDPP}. Then the assumptions of the theorem, together with Corollary \ref{cor:genAbundancereduction1} applied in dimension $n-1$, show that the divisor $K_X+tL$ is big by Lemma \ref{lem:maxnefdimBig} for every $t>n+1$. Therefore, we have $(K_X+tL)^n>0$, but then $K_X+tL$ is ample by \cite[Corollary 1.6]{CCP08} for every $t>n+1$.
\end{proof}

\section{Around the Semiampleness Conjecture}\label{sec:semiampleness}

The goal of this section is to reduce the Semiampleness Conjecture to the following substantially weaker version.

\begin{SemConVar}
Let $X$ be a projective variety with terminal singularities such that $K_X\sim 0 $. Let $L$ be a nef Cartier divisor on $X$.
Then there exists a semiample $\Q$-divisor $L'$ such that $L \equiv L'$.
\end{SemConVar}

We need the following result.

\begin{lem}\label{lem:quasietale}
Let $X$ be a projective variety with terminal singularities such that $K_X\equiv0$. Then there exists a quasi-\'etale cover $\pi\colon X'\to X$ such that $X'$ has terminal singularities and $K_{X'}\sim0$.
\end{lem}

\begin{proof}
By \cite[Theorem 8.2]{Kaw85b} we have $K_X\sim_\Q0$. Let $m$ be the smallest positive integer such that $mK_X\sim0$ and let $\pi\colon X'\to X$ be the corresponding $m$-fold cyclic covering. Then $\pi$ is quasi-\'etale and $K_{X'}\sim0$. 

It remains to show that $X'$ has terminal singularities. Let $E'$ be a geometric valuation over $X'$ which is not a divisor on $X'$. Then by \cite[Proposition 2.14(i)]{DL15} there exist a geometric valuation $E$ over $X$ which is not a divisor on $X$ and a positive integer $1\leq r\leq m$ such that
$$a(E',X',0)+1\geq r\big(a(E,X,0)+1\big).$$
Since $a(E,X,0)>0$, we have $r\big(a(E,X,0)+1\big)>1$, and hence $a(E',X',0)>0$, as desired.
\end{proof}

The following theorem shows that, modulo the Minimal Model Program, the Semiampleness Conjectures reduces to the case where $X$ has terminal singularities and $\Delta=0$.

\begin{thm}\label{thm:reductiontoSemiamplenessCY}
Assume the termination of flips in dimensions at most $n$ and the Abundance Conjecture in dimensions at most $n$. Then the Semiampleness Conjecture on Calabi-Yau varieties in dimension $n$ implies the Semiampleness Conjecture in dimension $n$.
\end{thm}

\begin{proof}
\emph{Step 1.}
Let $(X,\Delta)$ be a projective klt pair such that $K_X+\Delta\equiv 0$ and let $L$ be a nef Cartier divisor on $X$. Assume first that there exists a positive rational number $\mu$ such that $\mu\Delta+L$ is num-effective. We claim that then $L$ is num-semiample. Indeed, let $G$ be an effective $\Q$-divisor such that
\begin{equation}\label{eq:121}
\mu\Delta+L\equiv G. 
\end{equation} 
Pick a rational number $0<\delta\ll1$ such that the pair $\big(X,(1-\delta)\Delta+\frac{\delta}{\mu}G\big)$ is klt. Then by \eqref{eq:121} we have
\begin{align*}
0&\equiv K_X+\Delta=K_X+(1-\delta)\Delta+\delta\Delta\\
&\textstyle =K_X+(1-\delta)\Delta+\frac{\delta}{\mu}G-\frac{\delta}{\mu}L,
\end{align*}
so that
$$\textstyle \frac{\delta}{\mu}L\equiv K_X+(1-\delta)\Delta+\frac{\delta}{\mu}G.$$
Then $K_X+(1-\delta)\Delta+\frac{\delta}{\mu}G$ is semiample by the Abundance Conjecture, and the claim follows.

\medskip

\emph{Step 2.}
Now, pick a rational number $0<\varepsilon\ll1$ such that the pair $\big(X,(1+\varepsilon)\Delta\big)$ is klt, and denote $\Delta':=(1+\varepsilon)\Delta$. Observe that
\begin{equation}\label{eq:deltadeltaprime}
K_X+\Delta'\equiv\varepsilon\Delta.
\end{equation}
Assume first that $\nu(X,\Delta)>0$. Then $\nu(X,K_X+\Delta')>0$ by \eqref{eq:deltadeltaprime}, hence the divisor $K_X+\Delta'+L$ is num-effective by Theorem \ref{main_theorem1}(i). As
$$\varepsilon\Delta+L\equiv K_X+\Delta'+L$$
by \eqref{eq:deltadeltaprime}, the divisor $L$ is num-semiample by Step 1.

\medskip

\emph{Step 3.}
From now on we assume, by \eqref{eq:deltadeltaprime}, that $\nu(X,K_X+\Delta')=\nu(X,\Delta)=0$. As explained in \S\ref{subsec:scaling}, since we are assuming the termination of flips in dimension $n$, there is a $(K_X+\Delta')$-MMP
$$\varphi\colon (X,\Delta')\dashrightarrow (X_{\min},\Delta'_{\min})$$ 
and there exist a positive rational number $\lambda$ and a divisor $L_{\min}:=\varphi_*L$ on $X_{\min}$ such that:
\begin{enumerate}
\item[(a)] $K_{X_{\min}}+\Delta'_{\min}$ is nef and $s(K_{X_{\min}}+\Delta'_{\min})+L_{\min}$ is nef for all $s\geq\lambda$,
\item[(b)] the map $\varphi$ is $\big(s(K_X+\Delta')+L\big)$-negative for $s>\lambda$.
\end{enumerate}
Since $\nu(X_{\min},K_{X_{\min}}+\Delta'_{\min})=\nu(X,K_X+\Delta')=0$, by \eqref{eq:deltadeltaprime} we must have
$$\varepsilon\varphi_*\Delta\equiv K_{X_{\min}}+\Delta'_{\min}\sim_\Q 0,$$
hence $\Delta'_{\min}=(1+\varepsilon)\varphi_*\Delta=0$ and therefore $K_{X_{\min}}\sim_\Q 0$. 

\medskip

\emph{Step 4.}
By \cite[Proposition 2.18]{GGK17} there exists a quasi-\'etale cover $\pi_1\colon X_1\to X_{\min}$ such that $X_1$ has canonical singularities and $K_{X_1}\sim 0$. If $\pi_2\colon X_2\to X_1$ is a terminalisation of $X_1$ (see \cite[Corollary 1.4.3]{BCHM} and the paragraph after that result), then $X_2$ has terminal singularities and $K_{X_2}\sim_\Q0$. Then by Lemma \ref{lem:quasietale} there exists a quasi-\'etale cover $\pi_3\colon X'\to X_2$ such that $X'$ has terminal singularities and $K_{X'}\sim0$. Denote $\pi:=\pi_1\circ\pi_2\circ\pi_3\colon X'\to X$.

Fix $\lambda_0>\lambda$. Then the divisor 
$$\lambda_0 K_{X'}+\pi^*L_{\min}\sim_\Q\pi^*\big(\lambda_0(K_{X_{\min}}+\Delta_{\min})+L_{\min}\big)$$
is nef by (a). By the Semiampleness Conjecture on Calabi-Yau varieties the divisor $\lambda_0 K_{X'}+\pi^*L_{\min}$ is num-effective, hence so is $\lambda_0(K_{X_{\min}}+\Delta_{\min})+L_{\min}$ by Lemmas \ref{lem:descenteff1} and \ref{lem:descenteff2}. By \eqref{eq:deltadeltaprime}, by (b) and by Lemma \ref{lem:descenteff1} again, this then implies that the divisor 
$$\lambda_0\varepsilon\Delta+L\equiv\lambda_0(K_X+\Delta')+L$$ 
is num-effective. But then $L$ is num-semiample by Step 1.
\end{proof} 

\section{Surfaces and threefolds}\label{sec:lowdimensions}

In this section we prove Corollary \ref{cor:dim2} and discuss the state of the Semiampleness Conjecture on threefolds.

\subsection{Surfaces}

We need the following lemma.

\begin{lem}\label{lem:surfacestoabelian}
Let $(X,\Delta)$ be a projective klt surface pair such that $K_X+\Delta$ is pseudoeffective, and let $L$ be a nef $\Q$-divisor on $X$ such that $K_X+\Delta+L$ is nef. Let $\alpha\colon X\to A$ be a nontrivial morphism to an abelian variety $A$. Then $K_X+\Delta+tL$ is num-effective for all $t\geq0$.
\end{lem}

\begin{proof}
Since $\kappa(X,K_X+\Delta)\geq0$ by the Abundance Theorem on surfaces, it is enough to prove the lemma for any $t\geq1$. Fix any such $t$. By replacing $L$ by $tL$, we may assume that $t=1$.

If $n(X,K_X+\Delta+L)=0$, then $K_X+\Delta+L\equiv0$. If $n(X,K_X+\Delta+L)=1$, then the nef reduction $\varphi\colon X\dashrightarrow C$ of $K_X+\Delta+L$ to a curve $C$ is a morphism by \cite[\S2.4.4]{BCE+}. Hence by \cite[Proposition 2.11]{BCE+} there exists a nef $\Q$-divisor $P$ on $C$ such that $K_X+\Delta+L\equiv \varphi^*P$, and the lemma follows in this case.

Finally, assume that $n(X,K_X+\Delta+L)=2$. Consider the Stein factorisation $\alpha'\colon X\to A'$ of $\alpha$. If $\alpha'$ is birational, then $K_X+\Delta+L$ is $\alpha$-big. If $A'$ is a curve, then $(K_X+\Delta+L)|_F$ is ample for a general fibre $F$ of $\alpha'$ by the definition of the nef reduction, hence also in this case $K_X+\Delta+L$ is $\alpha$-big. The result then follows from \cite[Theorem 4.1]{BirChen15}.
\end{proof}

Now we can prove the Semiampleness Conjecture on surfaces.

\begin{thm}\label{thm:semiamplenessdim2}
Let $(X,\Delta)$ be a projective klt surface pair such that $K_X+\Delta\equiv0$ and let $L$ be a nef Cartier divisor on $X$. Then $L$ is num-semiample.
\end{thm}

\begin{proof}
It suffices to show that $L$ is num-effective: indeed, assume that there exists an effective $\Q$-divisor $G$ such that $L\equiv G$, and pick a rational number $0<\delta\ll1$ such that $(X,\Delta+\delta G)$ is klt. Then by the Abundance Theorem on surfaces, the divisor $K_X+\Delta+\delta G$ is semiample, hence $L$ is num-semiample since $\delta L\equiv K_X+\Delta+\delta G$.

There are several ways to show that $L$ is num-effective. First, by Theorem \ref{thm:reductiontoSemiamplenessCY} one can assume that $X$ is smooth and $\Delta=0$, in which case the result is known classically. Alternatively, when $\Delta\neq0$, there is an argument by distinguishing cases of a ruled surface and of a rational surface, see \cite[Lemma 4.2 and Example on p.\ 251]{Tot10}.

Here we give a different argument. By passing to a terminalisation of $(X,\Delta)$ (see \cite[Corollary 1.4.3]{BCHM} and the paragraph after that result), we may assume that the pair $(X,\Delta)$ is a terminal pair, and in particular, $X$ is smooth. By Lemma \ref{lem:surfacestoabelian} we may assume that $h^1(X,\OO_X)=0$, and hence $\chi(X,\OO_X)\geq1$. Then
$$\chi(X,L)=\frac12L\cdot(L-K_X)+\chi(X,\OO_X)=\frac12L\cdot(L+\Delta)+\chi(X,\OO_X)\geq1$$
since $L$ is nef and $\Delta$ is effective. We may assume that $L\not\equiv0$, hence $h^2(X,L)=h^0(X,K_X-L) = 0$ 
by Serre duality and this implies $h^0(X,L)\geq1$ as desired.
\end{proof}

Corollary \ref{cor:dim2} is an immediate consequence of Theorems \ref{main_theorem1}, \ref{main_theorem2} and \ref{thm:semiamplenessdim2}.

\subsection{Threefolds}

As mentioned in the introduction, the Semiampleness Conjecture is not known on threefolds. Here we reduce it further to the following conjecture.

\begin{con} \label{con:threefold} 
Let $X$ be a normal projective threefold with terminal singularities such that $K_X \sim 0$. Let $L$ be a nef Cartier divisor on $X$ such that either $\nu(X,L) = 1$, or $\nu(X,L) = 2$ and $L \cdot c_2(X) = 0$. Then $L$ is num-effective. 
\end{con}  

Note that the condition $L \cdot c_2(X) = 0$ can be rephrased by saying that $\pi^*L \cdot c_2(\widehat X) = 0$ for a desingularisation $\pi\colon \widehat X \to X$ which is an isomorphism on the smooth locus of $X$. For various results towards Conjecture \ref{con:threefold}, we refer to \cite{LOP16,LOP16a} and to \cite{LP18}. 

\begin{thm} \label{thm:threefold} 
Assume Conjecture \ref{con:threefold}. Let $(X,\Delta)$ be a $3$-dimensional projective klt pair such that $K_X + \Delta \equiv 0$ and let $L$ be a nef Cartier divisor on $X$. Then $L$ is num-semiample. 
\end{thm} 

\begin{proof} 
By Theorem \ref{thm:reductiontoSemiamplenessCY} we may assume that $X$ has terminal singularities and that $K_X \sim 0$. The result is clear if $\nu(X,L)\in\{0,3\}$, hence we may assume that $\nu(X,L) \in \{1,2\}$. By \cite{Og93,KMM94}, it suffices to show that $L$ is num-effective: this is shown as in the first paragraph of the proof of Theorem \ref{thm:semiamplenessdim2}. Since in any case $L\cdot c_2(X)\geq0$ by \cite{Miy87}, by Conjecture \ref{con:threefold} we may assume that $\nu(X,L) = 2$ and $L \cdot c_2(X) >0$.

By Serre duality we have $\chi(X,\OO_X)=0$, and thus, Hirzebruch-Riemann-Roch gives
$$ \chi(X,L) = \frac {1} {2} L \cdot c_2(X)>0.$$
As $h^q(X,L) = 0$ for $q\geq 2$ by the Kawamata-Viehweg vanishing (see for instance \cite[Lemma 2.1]{LP17}), we infer $h^0(X,L) >0$, which finishes the proof.
\end{proof}

As an immediate consequence, we obtain:

\begin{cor}
Suppose that Conjecture \ref{con:threefold} holds. Then the Generalised Nonvanishing Conjecture and the Generalised Abundance Conjecture hold in dimension three.
\end{cor}

\begin{proof}
This follows from Theorems \ref{main_theorem1}, \ref{main_theorem2} and \ref{thm:threefold}.
\end{proof}

\bibliographystyle{amsalpha}

\bibliography{biblio}
\end{document}